\documentclass[makeidx,fullpage]{article}
\usepackage[utf8]{inputenc}
\usepackage[T1]{fontenc}
\usepackage{amsmath}
\usepackage{amsthm}
\usepackage{amssymb}
\usepackage{mathrsfs}
\usepackage{dsfont}
\usepackage{stmaryrd}
\usepackage{amscd}
\usepackage{float}
\usepackage[a4paper]{geometry}
\usepackage[multiple]{footmisc}
\usepackage{lmodern}
\usepackage{pstricks}
\usepackage[colorlinks=true]{hyperref} 
\usepackage[capitalise]{cleveref}
\usepackage{verbatim}
\usepackage[all]{xy}
\usepackage{mathtools}
\usepackage{mathabx}
\usepackage{soul}

\newcommand{\vertiii}[1]{{\left\vert\kern-0.25ex\left\vert\kern-0.25ex\left\vert #1 
        \right\vert\kern-0.25ex\right\vert\kern-0.25ex\right\vert}}

\usepackage{array}

\title{\sf Existence of surfaces optimizing geometric and PDE shape functionals under reach constraint}
\author{ Yannick Privat\footnote{IRMA, Universit\'e de Strasbourg, CNRS UMR 7501, Inria, 7 rue Ren\'e Descartes, 67084 Strasbourg, France ({\tt yannick.privat@unistra.fr}).}~\footnote{Institut Universitaire de France (IUF).}
	\and R\'emi Robin\footnote{Laboratoire Jacques-Louis Lions, Sorbonne Universit\'e, Paris, France ({\tt remi.robin@inria.fr}).}
	\and Mario Sigalotti\footnote{Inria, France ({\tt mario.sigalotti@inria.fr}).}
}

\newcommand{\R}{\mathbb{R}}
\newcommand{\N}{\mathbb{N}}

\newcommand{\eps}{\varepsilon}

\newcommand{\Oo}{\mathscr{O}_{r_0}}

\newcommand{\smallo}{\operatorname{o}}
\newcommand{\bigO}{\operatorname{O}}

\DeclareMathOperator{\Id}{Id}

\makeatletter
\DeclareFontFamily{U}{tipa}{}
\DeclareFontShape{U}{tipa}{m}{n}{<->tipa10}{}
\newcommand{\arc@char}{{\usefont{U}{tipa}{m}{n}\symbol{62}}}%

\newcommand{\arc}[1]{\mathpalette\arc@arc{#1}}

\newcommand{\arc@arc}[2]{%
  \sbox0{$\m@th#1#2$}%
  \vbox{
    \hbox{\resizebox{\wd0}{\height}{\arc@char}}
    \nointerlineskip
    \box0
  }%
}
\makeatother

\renewcommand{\geq}{\geqslant}
\renewcommand{\leq}{\leqslant}

\newtheorem{theorem}{Theorem}

\newtheorem{proposition}{Proposition}
\newtheorem{corollary}{Corollary}
\newtheorem{definition}{Definition}
\newtheorem{lemma}{Lemma}
\theoremstyle{definition}
\theoremstyle{definition}\newtheorem{remark}{Remark}

\begin{document}
\maketitle
\begin{abstract}
  This article deals with the existence of hypersurfaces minimizing general shape functionals under certain geometric constraints.
  We consider as admissible shapes orientable hypersurfaces satisfying a so-called {\it reach} condition, also known as the uniform ball property, which ensures $\mathscr{C}^{1,1}$ regularity of the hypersurface.
  In this paper, we revisit and generalise the results of \cite{guoConvergenceBoundaryHausdorff2013,dalphinUniformBallProperty2018,dalphinExistenceOptimalShapes2020}.
  We provide a simpler framework and more concise proofs of some of the results contained in these references and extend them to a new class of problems involving PDEs. Indeed, by using the signed distance introduced by Delfour and Zolesio (see for instance \cite{delfourShapesGeometries2011}), we avoid the intensive and technical use of local maps, as was the case in the above references. Our approach, originally developed to solve an existence problem in \cite{privatOptimalShapeStellarators2022}, can be easily extended to costs involving different mathematical objects associated with the domain, such as solutions of elliptic equations on the hypersurface.
\end{abstract}
%\tableofcontents

\section{Framework and main results}
\subsection{Introduction}

In this paper, we are interested in the question of the existence of optimal sets for shape optimization problems involving surfaces. More precisely, we are interested in shape functionals written as
$$
  J(\Omega) = \int_{\partial \Omega} j(x,\nu_{\partial\Omega}(x), B_{\partial\Omega}(x)) \, d\mu_{\partial \Omega}(x)
$$
where $\Omega$ denotes a smooth subset of $\R^d$, the wording `smooth' being understood at this stage such that all the involved quantities make sense, $\nu$ denotes the outward pointing normal vector to $\partial\Omega$ and $B_{\partial\Omega}$ is either a purely geometric quantity such as the %so-called 
mean curvature, or the solution of a PDE on $\partial\Omega$ or on $\Omega$.

We are then interested in the existence of solutions for the optimization problem
$$
  \boxed{\inf_{\Omega \textrm{ admissible}%\in \mathcal{O}_{\rm ad}
  }J(\Omega)}.
$$
This kind of problem is very generic. What matters here is that the standard techniques, exposed and developed for example in \cite{delfourShapesGeometries2011,henrotShapeVariationOptimization2018}, do not apply to $d-1$ objects and it is necessary to adopt a particular approach. The first question to ask is the choice of the 
{set  $\mathcal{O}_{\rm ad}$ of all admissible domains}.
%admissible set $\mathcal{O}_{\rm ad}$. 
Since the shape functionals we consider involve geometric quantities of the type ``outward normal vector to the boundary'' or ``mean curvature'', it is necessary that the manipulated surfaces %cannot be 
are not too irregular. For this reason, we choose to impose a constraint that guarantees a uniform regularity, say $\mathscr{C}^{1,1}$, of the manipulated sets. 
%This also guarantees that the above quantities are defined at almost every point on the surface. 
This uniform regularity constraint is imposed by using the notion of ``reach''. Thus, the set $\mathcal{O}_{\rm ad}$ represents the set of surfaces having a reach uniformly bounded by below. The precise definition of this notion will be given %made precise 
in Section~\ref{sec:posReach}.

This kind of problem has been the subject of recent studies and results %progress, in 
\cite{guoConvergenceBoundaryHausdorff2013,dalphinUniformBallProperty2018,dalphinExistenceOptimalShapes2020}, which have provided positive answers to the existence issues.
In their approach, the authors used an efficient, but nevertheless laborious, %tedious 
approach based on the parametrization of the manipulated surfaces, seen as regular manifolds, using local charts.

The objective of this paper is to promote a different approach, based on the extension of the functions defined on the manipulated surfaces to volume neighborhoods, the introduction of an extruded surface and the rewriting of the surface integrals as volume integrals using ad-hoc variable changes. 
This is a methodological paper, in which a proof method is presented that may work in many cases. The results contained in the article illustrate this point. We discuss possible generalizations of these results in a concluding section. 

This method allows to gain conciseness and provides much shorter and direct existence proofs than in the above references. The method also allows to extend the field of investigation to new families of problems, involving the solution of a PDE defined on a hypersurface. Nevertheless, some arguments used by the authors of \cite{guoConvergenceBoundaryHausdorff2013,dalphinUniformBallProperty2018,dalphinExistenceOptimalShapes2020} cannot be shortened by using our approach. We have therefore chosen to expose our method in a short article, in which we detail all the parts of the proof that can be condensed and we %will 
make the necessary reminders concerning the results that cannot be condensed.

The article is organized as follows: we introduce the definition of the reach of a surface as well as the class of admissible sets we will deal with in Section~\ref{sec:posReach}. The main results of this article, regarding several existence results for shape optimization problems involving surfaces, are provided in Section~\ref{subsec:main_th}. The whole section \ref{sec:proofsTotal} is devoted to the proofs of the main results. In these proofs, we detail the arguments based on our approach and leading to simplified proofs of the results in \cite{guoConvergenceBoundaryHausdorff2013,dalphinUniformBallProperty2018,dalphinExistenceOptimalShapes2020}.
In order to illustrate the potential of our approach, we also provide an existence result involving a general functional depending on the solution of a PDE on the sought manifold.

\subsection{Notations}
Let us recall some classical %usual 
notations %we will use 
used
throughout this paper: % in shape optimization:
\begin{itemize}
  \item For the sake of notational simplicity, we will sometime use the notation $\Gamma$ (resp. $\Gamma_n$) to denote the hypersurfaces $\partial \Omega$ (resp. $\partial \Omega_n$).
  \item The Euclidean inner product (resp. norm) will be denoted $\langle\cdot,\cdot\rangle$ (resp. $\Vert \cdot\Vert$ or sometimes $|\cdot|$ when no confusion with other notations is possible). 
  %The operator norm on matrices associated with the Euclidean inner product will be denoted $\| \cdot \| $.  
  \item Given two positive integers $k\leq d$ and
        $\Omega \subset \R^d$,  $\mathcal{H}^k(\Omega)$ denotes the $k$-dimensional Hausdorff measure of $\Omega$.

  \item Given $\Omega \subset \R^d$, the distance (resp., signed distance) to $\Omega$ is defined for all $x\in \R^d$ by
        $$
          d_\Omega(x)= \inf_{y\in \Omega} \|x-y\|\qquad \text{(resp., }b_\Omega(x)=d_\Omega(x)-d_{\R^d \setminus\Omega}(x)).
        $$
  \item  Given $\Omega \subset \R^d$ and $h>0$, the tubular neighborhood $U_h(\Omega)$ is defined as
        $$U_h(\Omega)= \{ x \in \R^d \mid d_\Omega(x)\leq h \}.$$

  \item  Given $\Omega \subset \R^d$, the reach of $\Omega$ is defined as
        $$ \operatorname{Reach}(%\partial
        \Omega)=\sup \{ h>0 \mid d_\Omega \text{ is differentiable in } U_h(\Omega)\setminus \Omega \}.
        $$
        Recall that, if $\partial\Omega$ is a nonempty compact $\mathscr{C}^{1,1}$-hypersurface of $\R^d$, then there exists $h > 0$ such that $\Omega$ satisfies a uniform ball condition, namely
        \begin{equation}\label{ballCion}
          \tag{$\mathcal{B}_h$}
          \forall x\in \partial\Omega, \ \exists d_x\in \R^n \mid \Vert d_x\Vert_{\R^d}=1, \ B_h(x-hd_x)\subset \Omega \text{ and }B_h(x+hd_x)\subset \R^n\backslash \Omega,
        \end{equation}
        where $B_h(x)$ stands for the open ball of radius $h$ centered in $x$.
        Furthermore, assuming $\mathcal{H}^d(\partial \Omega)=0$, we have the simpler characterization
        $$
         \operatorname{Reach}(\partial          \Omega)=\sup \{h \mid \Omega \text{ satisfies \eqref{ballCion}} \}.
        $$
        Conversely, if $\partial
        \Omega$ is nonempty and satisfies Condition~\eqref{ballCion},
%        the condition above, 
then its reach is %positive  and 
larger than $h$ and the Lebesgue measure of $\partial\Omega$ in $\R^n$ is equal to 0. Furthermore, $\partial\Omega$ is a $\mathscr{C}^{1,1}$ hypersurface of $\R^n$. We refer for instance to Theorems 2.6 and 2.7 in \cite{dalphinUniformBallProperty2018}.
    \item For a given oriented $\mathscr{C}^{1,1}$-hypersurface $\partial \Omega$, we denote by $\nabla_{\partial \Omega}$ or $\nabla_\Gamma$ the tangential gradient and by $\nabla$ the full gradient in $\R^d$. When needed, each gradient will be assimilated to a line vector in $\R^d$.
  \item $\overline{\N}$ denotes $\N \cup \{ + \infty\}$.
  \item $\mathcal{S}^{d-1}$ denotes the unit sphere of $\R^d$.
  \item $M_d(\R)$ denotes the linear space of $d\times d$ matrices with real entries, endowed with the Euclidean operator norm $\| \cdot \| $. $\Id$ denotes the identity matrix in $\R^d$.
  \item For a given $\mathscr{C}^{1,1}$ hypersurface $\partial \Omega$, we denote by $H_{\partial \Omega}: \partial \Omega \to \R$, its mean curvature. We refer to Appendix \ref{sec:curv} for proper definitions.
\end{itemize}

\subsection{Preliminaries on sets of uniformly positive reach}\label{sec:posReach}

Given $r_0>0$ and a nonempty compact set $D\subset \R^d$, let us introduce the set $\Oo$ of admissible shapes whose  reach is bounded by $r_0$, namely
\begin{align}
  \label{def:admissible_set}
  \Oo=\{ \Omega \subset D \mid \Omega \text{ is closed}, \, \operatorname{Reach}(\partial 
  \Omega) \geq r_0 ,\, %\partial 
  \Omega \not = \emptyset,  \text{ and } \mathcal{H}^d(\partial \Omega)=0\}.
\end{align}
The elements of $\Oo$ are known to satisfy the following properties.
\begin{lemma}
  \label{lem:carac_Oo}
  Let $\Omega \in \Oo$. Then
  \begin{enumerate}
    \item \label{lem_prop:carac_Oo} $\partial \Omega$ is a $\mathscr{C}^{1,1}$ $(d-1)$-submanifold. Conversely,
    %\footnote{\ms{regularity of $\partial\Omega$?}}
          $$ \Oo=\{ \Omega \subset D \mid \Omega \text{ is closed}, \, \operatorname{Reach}(\partial 
          \Omega) \geq r_0 ,\, \partial \Omega \text{ is a $(d-1)$-submanifold}\}.$$
       \item \label{lem_prop:normal} For $x\in \partial \Omega$, $\nabla b_\Omega(x)$ is the unit outward normal vector.
    \item  \label{lem_prop:Lipschitz} For $h < r_0$, $\nabla b_\Omega$ is $\frac{2}{r_0-h}$-Lipschitz continuous on the tubular neighborhood $U_h(\partial \Omega)$.
    \item \label{lem_prop:Lipschitz_surface} The restriction of $\nabla b_\Omega$ to $\partial \Omega$ is $\frac{1}{r_0}$-Lipschitz continuous.

    \item \label{lem_prop:bounded_area} There exists a constant $C$ depending only on $d$, $r_0$, and $D$ such that $\mathcal{H}^{d-1}(\partial \Omega) \leq C$.

  \end{enumerate}
\end{lemma}
Points~\ref{lem_prop:carac_Oo} and \ref{lem_prop:normal} are
proved in \cite[Theorem 8.2, Chapter 7]{delfourShapesGeometries2011}. Points~\ref{lem_prop:Lipschitz} and \ref{lem_prop:Lipschitz_surface} are proved in \cite[Theorems~2.7 and~2.8]{dalphinUniformBallProperty2018}.
The proof of Point~\ref{lem_prop:bounded_area} is given in Section~\ref{subsubseq:proof_lem_prop:bounded_area}.

We will endow the set $\Oo$ with a `sequential' topology, by introducing a notion of convergence in this set.
\begin{definition}[$R$-convergence in $\Oo$]
  Given $(\Omega_n)_{n\in\N} \in \Oo^\N$, we say that $(\Omega_n)_{n\in \N}$ $R$-converges to $\Omega_\infty \in \Oo$ and we write $\Omega_n \xrightarrow{R} \Omega_\infty$ if %and only if 
  \begin{align}
    b_{\Omega_n} \to b_{\Omega_\infty}
    \quad \begin{cases}
      \text{in }\mathscr{C}(\overline{D}),                                                                                  \\
      \text{in }\mathscr{C}^{1,\alpha}({U_{r}(\partial \Omega_\infty)}), \, \forall r< r_{0}, \,  \forall \alpha \in [0,1), \\
      \text{weakly-star in } W^{2,\infty}({U_{r}(\partial \Omega_\infty)}), \, \forall r< r_{0}.
    \end{cases}
  \end{align}
\end{definition}

The next result justifies the interest of the class $\Oo$ endowed with the $R$-convergence for existence issues.

\begin{proposition}\label{th:compact}
 $\Oo$ is sequentially compact for the $R$-convergence.
\end{proposition}

%For the sake of readability, 
The proof of this proposition %is postponed to 
can be found in Section~\ref{sec:appendRCV}.
Let us end this section by providing several additional properties of the $R$-convergence.
\begin{lemma}
  \label{lem:seq_continuity}
  If $\Omega_n \xrightarrow{R} \Omega_\infty$ then
  \begin{enumerate}
    \item \label{lem:seq_continuity_perimeter} $\mathcal{H}^{d-1}(\partial \Omega_n)$ converges toward $\mathcal{H}^{d-1}(\partial \Omega_\infty)$ as $n\to +\infty$.
    \item \label{lem:seq_continuity_volume} $\mathcal{H}^{d}(\Omega_n)$ converges toward $\mathcal{H}^{d}(\Omega_\infty)$ as $n\to+\infty$.
    \item \label{lem:seq_continuity_homotopy}If all the $\partial \Omega_n$ belong to the same isotopic class, then $\partial \Omega_\infty$ also belongs such a class.
  \end{enumerate}
\end{lemma}

The proof of this lemma is given in Section~\ref{sec:proof-lem:seq_continuity}.

\begin{remark}
  According to Lemma \ref{lem:seq_continuity}, we obtain for example that for a given $\Omega_0\in \Oo$, and $a\leq b$,
  $$\{ \Omega \in \Oo \mid a\leq \mathcal{H}^{d-1}(\partial \Omega)\leq b , \, \partial \Omega \text{ is isotopic to } \partial \Omega_0 \}$$
  is a sequentially compact set.
\end{remark}

\subsection{Main results}
\label{subsec:main_th}

Let us introduce %$F_1$ 
the %following 
general shape functional
\begin{align}
  \label{F1}
  F_1(\Omega)=  & \int_{\partial \Omega} j_1(x,\nu(x),H_{\partial \Omega}(x)) \, d\mu_{\partial \Omega}(x),
\end{align}
where $j_1$ is continuous from $\R^d\times \mathcal{S}^{d-1}\times \R$ to $\R$ and convex with respect to its last variable. We %also 
recall that $\nu$ and $H_{\partial \Omega}$ denote respectively the outward pointing normal vector and the mean curvature.

According to \cref{th:compact}, the set $\Oo$ is sequentially compact for the $R$-convergence. Therefore, 
in order to infer the existence of an optimal surface minimizing $F_1$ over $\Oo$ it is enough to prove the lower semicontinuity of functional $F_1$ (under suitable assumptions on the function $j_1$). This is the main purpose of the following result.

\begin{theorem}[\cite{dalphinUniformBallProperty2018}, Theorem 1.3]  \label{th:lsc}
  Let us assume that $j_1$ is continuous with respect to all variables and convex with respect to its last one. %variable.
  Then, $F_1$ is a lower semi-continuous shape functional for the $R$-convergence, i.e., for every sequence $(\Omega_n)_{n\in\N} \in \Oo^\N$ that $R$-converges toward $\Omega_\infty$, one has
  \begin{align}
    \liminf_{n\to +\infty} F_1(\Omega_n) \geq F_1(\Omega_\infty).
  \end{align}
  As a consequence, the shape optimization problem
  $$
    \inf_{\Omega\in \Oo}F_1(\Omega)
  $$
  has a solution.
\end{theorem}
%Note that \cref{th:compact} and \cref{th:lsc} ensure existence of the shape optimization problem {\color{red} TO BE INTRODUCED}.
It is notable that, by applying \cref{th:lsc} both to $j_1$ and $-j_1$, we get the following corollary.

\begin{corollary}
  \label{cor:continuity}
  If $j_1$ is continuous and linear in the last variable, then $F_1$ is a continuous shape functionals for the $R$-convergence.
\end{corollary}

\begin{remark}
  \label{th:lsc2}
  In the case where $d=3$, it is proved in \cite[Theorem 1.3]{dalphinUniformBallProperty2018} that \cref{th:lsc} holds if we replace the mean curvature by the Gaussian one in the definition of $F_2$. We do not provide a proof here since most of the difficulties are related to the convergence of a product of weak-star converging sequences and our approach does not change the proof in a significant way.
\end{remark}
Let us now %provide two %examples\footnote{\ms{misleading in my opinion}} of cost functionals
consider two classes of shape optimization problems
 involving either an elliptic PDE inside $\Omega$ or an elliptic PDE on the $\mathscr{C}^{1,1}$ hypersurface $\partial \Omega$.

\paragraph{Problems involving an elliptic PDE on a $\mathscr{C}^{1,1}$-hypersurface of $\R^d$.}
%The next result deals with functionals involving the solution of a linear PDE on the manifold.
Given $f\in \mathscr{C}^0(D)$, we consider the problem of minimizing
%the evaluation of 
a shape functional depending on the
solution  $v_{\partial \Omega}$ of the equation
\begin{align}\label{PDEmanifold}
  \Delta_{\Gamma} v_{\partial \Omega}(x)=f(x) \quad \text{ in } \partial \Omega,
\end{align}
where $\Delta_{\partial \Omega}$ denotes the Laplace--Beltrami operator on $\partial \Omega$.
Since we are not considering $\mathscr{C}^\infty$ manifolds but rather $\mathscr{C}^{1,1}$ ones, we need to explain how the PDE must be understood. We use here an energy formulation, defining, for a closed and nonempty hypersurface $\partial\Omega$, the functional
\begin{equation}\label{def:NRJ}
  \mathscr{E}_{\partial \Omega}: H^1_*(\partial \Omega)\ni u\mapsto \frac12 \int_{\partial \Omega}|\nabla_{\Gamma} u(x)|^2d\mu_{\partial\Omega}-\int_{\partial \Omega}f(x)u(x)d\mu_{\partial\Omega}
\end{equation}
where $H^1_*(\partial \Omega)$ denotes the Sobolev space of functions in $H^1(\partial \Omega)$ with zero mean on $\partial \Omega$. We hence define $v_{\partial\Omega}$ as the unique solution of the minimization problem
\begin{equation}\label{minNRJEpOm}
  \min_{u\in H^1_*(\partial \Omega)}\mathscr{E}_{\partial \Omega}(u).
\end{equation}

\begin{lemma}\label{lem:1915}
  Let $\Omega\in \Oo$. Problem \eqref{minNRJEpOm} has a unique solution $v_{\partial\Omega}$. Furthermore, if $\partial\Omega$ is $\mathscr{C}^2$ and if $f\in \mathscr{C}^0(D)$, then $v_{\partial\Omega}$ satisfies \eqref{PDEmanifold} almost everywhere in $\partial\Omega$.
\end{lemma}
%For the sake of clarity, 
The proof of this result is postponed to Section~\ref{proof:lem:1915}.

Let us introduce the shape functional
$$
  F_2(\Omega)=\int_{\partial \Omega} j_2(x,\nu(x),v_{\partial \Omega}(x),\nabla_{\Gamma} v_{\partial \Omega}(x)) \, d\mu_{\partial \Omega}(x),
$$
where $j_2 : \R^d \times \mathcal{S}^{d-1} \times \R \times \R^d \to \R$ is assumed to be continuous.

\begin{theorem}
  \label{th:elliptic_in_boundary}
  The shape functional $F_2$ is  lower semi-continuous for the $R$-convergence, i.e., for every sequence $(\Omega_n)_{n\in\N} \in \Oo^\N$ that $R$-converges toward $\Omega_\infty$, one has
  \begin{align}
    \liminf_{n\to +\infty} F_2(\Omega_n) \geq F_2(\Omega_\infty).
  \end{align}
  As a consequence, the shape optimization problem
  $$
    \inf_{\Omega\in \Oo}F_2(\Omega)
  $$
  has a solution.
\end{theorem}

\paragraph{Problems involving an elliptic PDE in a domain of $\R^d$.}
Finally, let us investigate the case of a shape criterion involving the solution of a PDE on a domain of $\R^d$.
We consider hereafter a Poisson equation with non-homogeneous boundary condition, but we claim that all conclusions can be easily extended to a larger class of elliptic PDEs.

Let $h\in L^2(D)$, $g\in H^2(D)$, and define $u_\Omega$ as the solution of
\begin{align}
  \label{eq:PDE_Omega}
  \left\{\begin{array}{ll}
    \Delta u_\Omega= h & \text{in $\Omega$},          \\
    u_\Omega= g        & \text{in $\partial \Omega$}.
  \end{array}\right.
\end{align}

Let us introduce the shape functional $F_3$ given by
$$
  F_3(\Omega)=\int_{\partial \Omega} j_3(x,\nu(x),u_{\Omega}(x),\nabla u_{\Omega}(x)) \, d\mu_{\partial \Omega}(x),
$$
where $j_3 : \R^d \times \mathcal{S}^{d-1} \times \R %\R^d 
\times \R^d 
\to \R$ is continuous.
\begin{theorem}[\cite{dalphinExistenceOptimalShapes2020}, Theorem~2.1]
  \label{th:elliptic_inside}
  The shape functional  $F_3$ is lower semi-continuous for the $R$-convergence.
\end{theorem}
We mention this theorem demonstrated in \cite{dalphinExistenceOptimalShapes2020}. Nevertheless, it is notable that by adapting the proof of Theorem~\ref{th:elliptic_in_boundary}, it is possible to obtain a much shorter proof of this theorem. In order not to make this article unnecessarily heavy, we only give the main steps of the proof in Section~\ref{sec:proof:th:elliptic_inside}.
%, the detail of the steps being left as an exercise to the reader. 
This example is mentioned %, on the one hand 
both for the sake of completeness, in order to review the existing literature, %but 
and also to underline the potential of the approach introduced here, which allows %in particular 
to find 
more direct proofs of
all the known results and to extend them. %, with the help of more direct proofs.

In addition, %particular, 
it is interesting to notice that our %the 
approach allows to deal with problems involving PDEs both using weak %variational
 formulations as in \eqref{eq:PDE_Omega} and also %with PDEs 
whose solutions are obtained using a minimization principle, as is the case in \eqref{PDEmanifold}. The approach thus seems robust and we believe that it can be easily adapted to general families of problems (for example to a general non-degenerate elliptic PDE).

%\rr{Je pense qu'on pourrait mettre un petit mot (ici ou dans l'intro) pour dire qu'on va utiliser la formulation faible pour résoudre ce problème. Ainsi notre méthode est adaptée à la fois pour des méthodes variationelles (partie de Yannick sur l'edp sur la surface, très utile pour des pb non linéaire par exemple) et des approches par formulation faible.}

\section{Proofs}\label{sec:proofsTotal}

\subsection{The extruded surface approach}
\label{sec:extruded_surf}

One of the key ideas to prove sequential continuity of functionals involving an integral on the boundary is to approximate such an integral by an integral on a small tubular neighborhood (as done e.g. in \cite{delfourTangentialDifferentialCalculus2000}).

Let us first illustrate the method by proving Point~\ref{lem_prop:bounded_area}
of \cref{lem:carac_Oo}.

\subsubsection{Proof of \Cref{lem:carac_Oo}, Point \ref{lem_prop:bounded_area}}
\label{subsubseq:proof_lem_prop:bounded_area}

For $0<h <r_0$, consider
\begin{equation}
  T:
  \begin{array}[t]{rcl}
    (-h,h) \times \partial \Omega & \to     & U_h(\partial \Omega)       \\
    (t,x)                         & \mapsto & x+t \nabla b_{ \Omega}(x).
  \end{array}
\end{equation}
Since $T$ is Lipschitz continuous, it is  differentiable at almost every $(t_0,x_0)$, with
\begin{align}
  \label{eq:dT}
  d_{(t_0,x_0)}T(s,y)=y+s \nabla b_{ \Omega}(x_0) + t_0 d_{x_0}\nabla b_{ \Omega}(y), \qquad \forall (s,y) \in \R \times T_{x_0} \partial \Omega.
\end{align}

\begin{remark}
  \label{rmk:identification_tangent_plan}
  Note that as $\nabla b_{ \Omega}(x_0)$ is a normal unit vector to $\partial \Omega$ at $x_0$, we can identify the tangent hyperplane $T_{x_0} \partial \Omega$ with $\R^{d-1}$ endowed with an Euclidean structure inherited from that of $\R^d$. We will use this identification several times in this paper.
\end{remark}
As a result, we can identify $ \R \times T_{x_0} \partial \Omega\ni (s,y) \mapsto y+s \nabla b_{ \Omega}(x_0)$ with an orthogonal matrix.
Moreover, up to the choice of a different orientation on $T_{x_0} \partial \Omega$,
such a matrix belongs to the special orthogonal group ${\rm SO}(n)$.
We use the same coordinate representation to identify $\R \times T_{x_0} \partial \Omega\ni (s,y) \mapsto d_{x_0}\nabla b_{ \Omega}(y)$ with a $n\times n$ matrix.
%, we can consider the determinant of this application. A we have that 
%$|\det [s,y \to y+s \nabla b_{ \Omega}(x_0)]|=1$. 
%\footnote{\rr{j'ai l'impression qu'il faut changer l'ordre t,x pour avoir la bonne orientation}}
%
By uniform continuity of the determinant %(in dimension $d$) 
around ${\rm SO}(d)%Id_{\R^d}
$, there exists $C_0>0$ such that, for every $M\in {\rm SO}(d)$ and every $l\in M_d(\R)$ such that $\|l\|\leq C_0$,
\[\frac12\le \det (M +l) \le \frac32.\]
%$| |\det (M +l)| -1| \leq \frac{1}{2}$.
As $\nabla b_\Omega$ is $\frac{2}{r_0}$-Lipschitz continuous on $\partial \Omega$, we have that for almost every $x_0\in \partial \Omega$ and every $t_0\in\R$, $\|t_0 d_{x_0}\nabla b_{ \Omega}\| \leq \frac{2 |t_0|}{r_0}$.

Let us fix $h<\min(r_0,r_0 C_0 /2)$ (independent of $\Omega$), so that
%$U_h(\partial \Omega)\subset D$ and 
$\|t_0 d_{x_0}\nabla b_{ \Omega}\| \leq C_0$ for almost every $x_0\in \partial \Omega$ and every $t_0\in (-h,h)$. By the change of variable formula we then have
$$\mathcal{H}^{d-1}(\partial \Omega)=\int_{\partial \Omega} d\mu_{\partial \Omega}= \frac{1}{2h}\int_{U_h(\partial \Omega)} \det (d_{T^{-1}(y)}T) dy
  %$$
  %Hence 
  %$$\mathcal{H}^{d-1}(\Gamma)
  \leq \frac{3}{4h} \mathcal{H}^{d}(U_h(D)),$$
whence the conclusion.

\subsubsection{Extruded surface and $R$-convergence}
\label{subsec:extruded_surf}

Let us now illustrate the power of this approach in the case of a $R$-converging sequence.

Let $\Omega_n \xrightarrow{R} \Omega_\infty$. From now on, we will use the notation $\Gamma_n \coloneqq \partial \Omega_n$ for the hypersurfaces.

For $h<r_0$ and $n \in \overline{\N}$, let us define a parametrization of a neighborhood of $\Gamma_n$ by
\begin{equation}
  T_n:
  \begin{array}[t]{rcl}
    (-h,h) \times \Gamma_n & \to     & U_h(\Gamma_n)         \\
    (t,x)                           & \mapsto & x+t \nabla b_{ \Omega_{n}}(x).
  \end{array}
\end{equation}

\begin{lemma}
  \label{lem:det_bounded}
  For every $\eps>0$, there exists $h>0$ such that for all $n \in  \overline{\N}$,
  \begin{align*}
    1-\eps \leq \det (d_{(t_0,x_0)}T_n)\leq 1+\eps,\quad \text{for a.e. }(t_0,x_0)\in
    (-h,h)
    \times \Gamma_n.
  \end{align*}
\end{lemma}
\begin{proof}
  We follow the same argument as in Section~\ref{subsubseq:proof_lem_prop:bounded_area}. Namely, for a given $\eps>0$, there exists $C_0>0$ such that for every $M\in {\rm SO}(d)$ and every $l\in M_d(\R)$ such that $\|l\|\leq C_0$,
  \[ 1-\eps \leq \det (M +l) \leq 1+\eps.\]
  Let us fix $h<\min(r_0,r_0 C_0 /2)$ (independent of $n$). As $\nabla b_{\Omega_n}$ is $\frac{2}{r_0}$-Lipschitz continuous on $\Gamma_n$, we get
  %$U_h(\Gamma)\subset D$ and 
  $\|t_0 d_{x_0}\nabla b_{ \Omega_n}\| \leq C_0$ for almost every $x_0\in \Gamma$ and every $t_0\in (-h,h)$. Whence, using  Equation~\eqref{eq:dT} with the previous estimate, we conclude the proof.
\end{proof}

\begin{remark}
  In what follows, we will use the Bachmann--Landau notation $\operatorname{o}_{h\to 0}(1)$%like 
   for a function 
  %whose 
  converging to $0$ in $L^\infty$ as $h$ goes to $0$ and for a given $n$, large enough. For example, \cref{lem:det_bounded} %states 
  implies that
  $$\det (dT_n)=1 + {\operatorname{o}}_{h\to 0}(1),\qquad \mbox{on }(-h,h)\times \Gamma_n,$$
  which means
  $\forall \eps>0, \exists N_0\in \N,\, \exists h>0, \forall n \in \bar{\N}, \, n\geq N_0$ implies
  $$ \left| \det (d_{(t_0,x_0)}T_n) -1 \right| \leq \eps,\qquad \mbox{for a.e. } (t,x)\in (-h,h)\times \Gamma_n.$$
  %\ms{(One actually could take $N_0=1$ in this example.)}
\end{remark}

Let us now introduce the orthogonal projection  $p_n$ onto $\Gamma_n$, defined on $U_h(\Gamma_n)$ for every $h\in (0,r_0)$.
\begin{lemma}
  The following properties hold:
  \label{lem:prop_p}
  \begin{enumerate}
    \item \label{lem:prop_p_inverse} $p_n$ 
    coincides with the 
 %   is the projection on the 
    second component of $T_n^{-1}:U_h(\Gamma_n)\to (-h,h)\times \Gamma_n$.
%    the inverse of $T_n$.
    \item \label{lem:prop_p_proj} For all $x\in U_h(\Gamma_n)$, $p_n(x)=x-b_{\Omega_n}(x) \nabla b_{\Omega_n}(x)$.
    \item \label{lem:prop_p_convergence}
          %For all $h< r_0$,
          $p_n$ converges toward $p_\infty$ in $L^\infty(U_h(\Gamma_\infty))$.
  \end{enumerate}
\end{lemma}
\begin{proof}
  \Cref{lem:prop_p_inverse,lem:prop_p_proj} are obviously equivalent and are proved in \cite[Theorem~7.2, Chapter~7]{delfourShapesGeometries2011}. \Cref{lem:prop_p_convergence} follows from the $\mathscr{C}^1$ convergence of $b_{\Omega_n}$ toward $b_{\Omega_\infty}$.
\end{proof}
We can now state the key equality to relate surface and volume integrals. 
Apply Lemma~\ref{lem:det_bounded} with $\eps\in (0,1)$ to select $h>0$ such that $T_n:(-h,h)\times \Gamma_n\to U_h(\Gamma_n)$ is invertible for every $n\in \overline{\N}$\footnote{It is actually well-known that the domain of invertibility of $T_n$ contains $U_{r_0}(\Gamma_n)$.}.
%Let us fix $\eps\in (0,1)$ and take $h$ given by Lemma \ref{lem:det_bounded} to ensure the invertibility of $T$\footnote{It is in fact well-known that the domain of invertibility of $T_n$ contains $U_{r_0}(\Gamma_n)$.}.
\begin{lemma}
  \label{lem:integral_volume_surface}
  For all $n\in \overline{\N}$, $f\in L^1(\Gamma_n)$, and $t \in (0,h)$ we have
  \begin{align*}
    \int_{\Gamma_n} f(x)\, d\mu_{\Gamma_n}(x)= \frac{1}{2t}\int_{U_t(\Gamma_n)} f\circ p_n(y) \,  \det(d_{T_n^{-1}(y)}T_n) \, dy.
  \end{align*}
\end{lemma}

\begin{proof}
  Using  the change of variable formula (also known as area formula for Lipschitz continuous functions), one gets
  \begin{align*}
    %\label{eq:integral_volume_surface}
    \int_{-t}^\top  \int_{\Gamma_n} f(x)\, d\mu_{\Gamma_n}(x) dt= \int_{U_t(\Gamma_n)}f\circ p_n(y)\det(d_{T_n^{-1}(y)}T_n) \, dy.
  \end{align*}
\end{proof}

From now on, we will %drop 
the $T_n^{-1}(y)$ inside the determinant to improve the readability.

\begin{lemma}
  \label{lem:inclusions}
  For every $h\leq r_0/2$ and $0<t<h$, there exists $N_0$ such that
  $$\forall n \geq N_0, \quad U_{h-t}(\Gamma_\infty)\subset U_{h}(\Gamma _n) \subset U_{h+t}(\Gamma_\infty).$$
\end{lemma}
\begin{proof}
  By uniform convergence of $b_{\Omega_n}$ toward $b_{\Omega_\infty}$, we have that for $n$ large enough
  $$b_{\Omega_\infty}^{-1}((t-h, h-t)) \subset b_{\Omega_n}^{-1}((-h, h)) \subset b_{\Omega_\infty}^{-1}((-h-t, h+t)).
  $$
\end{proof}

In order to perform changes of variable in surface integrals, it is convenient to use directly $p_n$ as a way to map $\Gamma_\infty$ onto $\Gamma_n$. To this aim, we define
\begin{equation}
  \tau_n:
  \begin{array}[t]{rcl}
    \Gamma_\infty & \to     & \Gamma_n         \\
    x                           & \mapsto & p_n(x).
  \end{array}
\end{equation}
Note that for $n$ large enough, \cref{lem:inclusions} ensure that $\tau_n$ is well-defined. We also introduce $\operatorname{Jac}(\tau_n)$ to denote the Jacobian of $\tau_n$. Then we have the following lemma.
\begin{lemma}
  \label{lem:p_n_det}
  For $n$ large enough, $\tau_n: \Gamma_\infty \to \Gamma_n$ is a diffeomorphism. Besides
  \begin{align}
    \label{eq:limit_det}
     \sup_{x\in \Gamma_\infty} \left|\operatorname{Jac}(\tau_n)(x)-1 \right| \xrightarrow[n \to \infty]{}0.
  \end{align}
\end{lemma}
\begin{proof}
  Let $x\in \Gamma_\infty $. We take $v \in T_x \Gamma_\infty$, and identify it with an element of the tangent hyperplane (see \cref{rmk:identification_tangent_plan}). As $v$ is tangent to $\Gamma _\infty$ at $x$, we get
  $$\langle v, \nabla b_{\Omega_\infty}(x)\rangle=0.$$
Using Lemma~\ref{lem:prop_p}, Item~\ref{lem:prop_p_proj}, 
%  Looking at the differential of $p_n$, 
  we get,
  $$d_x p_n(v)=v - \langle \nabla b_{\Omega_n}(x), v \rangle \nabla b_{\Omega_n}(x) - b_{\Omega_n}(x) \nabla^2 b_{\Omega_n}(x) v.$$
  Let us now fix $h<\frac{r_0}{3}$. For $n$ large enough, thanks to \cref{lem:inclusions}, we have $\Gamma_n \subset U_h(\Gamma_\infty)$. Thus,
  \begin{align*}
    \|d_x p_n(v)-v\| & \leq \|\nabla b_{\Omega_n}(x)-\nabla b_{\Omega_\infty}(x)\| \|v\| + \|b_{\Omega_n}\|_{L^\infty(\Gamma_\infty)} \|\nabla^2 b_{\Omega_n}(x)\| \|v\|                                                                                                                \\
                     & \leq \|v\| \left(\|\nabla b_{\Omega_n}-\nabla b_{\Omega_\infty}\|_{L^\infty(U_{\frac{r_0}{3}}(\Gamma_\infty))} + \|b_{\Omega_n}\|_{L^\infty(\Gamma_\infty)} \|\nabla^2 b_{\Omega_n}(x)\|_{L^\infty(U_{\frac{r_0}{3}}(\Gamma_\infty))} \right).
  \end{align*}
  We recall that both $\|\nabla b_{\Omega_n}-\nabla b_{\Omega_\infty}\|_{L^\infty(U_{\frac{r_0}{3}}(\Gamma_\infty))}$ and $\|b_{\Omega_n}\|_{L^\infty(\Gamma_\infty)}$ converge toward zero. Besides, the quantity $\|\nabla^2 b_{\Omega_n}(x)\|_{L^\infty(U_{\frac{r_0}{3}}(\Gamma_\infty))}$ is uniformly bounded. As a consequence,
  \begin{align}
    \sup_{x\in \Gamma_\infty} \sup_{\substack{v\in T_x \Gamma_n \\ \|v\|=1}}\| d_xp_n(v)-v \|\xrightarrow[n\to \infty]{} 0.
  \end{align}
  Using a  similar argument to the one used in \cref{lem:det_bounded}, we take the determinant and obtain \cref{eq:limit_det}.

  As a result we know that $\tau_n$ is a local diffeomorphism. It remains to prove that $\tau_n$ is injective. To this aim, we suppose that $n$ is large enough to ensure that
  $$\|\nabla b_{\Omega_n}-\nabla b_{\Omega_\infty}\|_{L^\infty(U_{\frac{r_0}{3}}(\Gamma_\infty))}< \frac{1}{2}.$$

  Let $x,y \in \Gamma_\infty$ such that $p_n(x)=p_n(y)$. If $x \not = y$, it implies that there exists $t \in (-\frac{2r_0}{3},\frac{2r_0}{3})\setminus \{0\}$  such that
  $$x=y+t\nabla b_{ \Omega_n}(p_n(y))=y+t\nabla b_{ \Omega_n}(y)$$
  As $\Omega_\infty \in \Oo$, it satisfies the $r_0$ uniform ball property (see \eqref{ballCion}). Thus, one has
  $$
    B_{r_0}\left(y+ r_0 \operatorname{sign} t \nabla b_{\Omega_\infty}(y)\right)\cap \Gamma_\infty= \emptyset.$$
  But we have
  \begin{align*}
    \left|x-y- r_0 \operatorname{sign} t \nabla b_{\Omega_\infty}(y)\right| & =\left| t \nabla b_{\Omega_n}(y)-r_0 \operatorname{sign} t \nabla b_{\Omega_\infty}(y)\right| \\
    & \leq \frac{t}{2} + \left|t-r_0 \operatorname{sign} t\right|< r_0.
  \end{align*}
  This is a contradiction, hence $\tau_n$ is injective which implies that it is a diffeomorphism from $\Gamma_\infty$ to $\Gamma_n$.
\end{proof}

\subsection{Proof of Lemma \ref{lem:seq_continuity}}\label{sec:proof-lem:seq_continuity}
Suppose that $(\Omega_n)_{n\in \N}\in \Oo^\N$ $R$-converges toward $\Omega_\infty \in \Oo$.
\subsubsection{Proof of Point~\ref{lem:seq_continuity_perimeter}}
For $h<r_0$ and using Lemma \ref{lem:integral_volume_surface}, we have
$$\mathcal{H}^{d-1}(\Gamma_n)= \int_{\Gamma_n} %1 \, 
  d\mu_{\Gamma_n}(x)= \frac{1}{2h}\int_{U_h(\Gamma_n)} %1 \, 
  \det(dT_n) \, dy.$$
By Lemma \ref{lem:inclusions}, moreover,
\begin{align*}
  \mathcal{H}^{d-1}(\Gamma_n)
  %&=\frac{1}{2h}\int_{U_h(\Gamma_n)} %1 \, 
  %\det(dT_n) \, dy \\
   & = \frac{1}{2h}\int_{U_{h-t}(\Gamma_\infty)} %1 \, 
  \det(dT_n)\, dy
  + \frac{1}{2h}\int_{U_h(\Gamma_n) \setminus U_{h-t}(\Gamma_\infty)} %1 \, 
  \det(dT_n) \, dy
\end{align*}
for $t\in (0,h)$ and $n$ large enough.
Let us compare the first term in the right-hand side with
$$\mathcal{H}^{d-1}(\Gamma_\infty)=\frac{1}{2(h-t)}\int_{U_{h-t}(\Gamma_\infty)} %1 \,  
  \det(dT_\infty) \, dy.$$
Using Lemma \ref{lem:det_bounded}, $\det(dT_\infty)=\det(dT_n)+{\operatorname{o}}_{h\to 0}(1)$ on $(-h,h)\times\Gamma_\infty$. Besides, $\frac{1}{2h}- \frac{1}{2(h-t)}= \operatorname{O}\left(\frac{t}{h}\right)$. Hence, % we get
$$\frac{1}{2h}\int_{U_{h-t}(\Gamma_\infty)} %1 \,  
  \det(dT_n) \, dy = \mathcal{H}^{d-1}(\Gamma_\infty) + {\operatorname{o}}_{h\to 0}(1)+\operatorname{O}\left(\frac{t}{h}\right).$$
On the other hand, using again the relation $\det(dT_\infty)=\det(dT_n)+{\operatorname{o}}_{h\to 0}(1)$,
\begin{align*}
  \frac{1}{2h}\int_{U_h(\Gamma_n) \setminus U_{h-t}(\Gamma_\infty)} %1 \,  
  \det(dT_n) \, dy & \leq \frac{1}{2h}\int_{U_{h+t}(\Gamma_\infty) \setminus U_{h-t}(\Gamma_\infty)}                                                                                   %1 \,  
  \det(dT_n) \, dy                                                                                                                                                                                       \\
  %  &= \frac{1}{2h}\int_{U_{h+t}(\Gamma_\infty) \setminus U_{h-t}(\Gamma_\infty)}
  %\det(dT_n) +{\operatorname{o}}_{h\to 0}(1) \, dy \\
                   & =\frac{1}{2h} \big( \int_{U_{h+t}(\Gamma_\infty)}                                                                                                                          %1 \, 
  \det(dT_n) \, dy -\int_{U_{h-t}(\Gamma_\infty)} %1 \,  
  \det(dT_n) \, dy \big)                                                                                                                                                                                 \\
                   & = \frac{1%+{\operatorname{o}}_{h\to 0}(1)
                   }{2h} (2(h+t) \mathcal{H}^{d-1}(\Gamma_\infty)- 2(h-t) \mathcal{H}^{d-1}(\Gamma_\infty)+{\operatorname{o}}_{h\to 0}(h))                          \\
                   & =%(1+{\operatorname{o}}_{h\to 0}(1))
                   \frac{2t}{h} \mathcal{H}^{d-1}(\Gamma_\infty)+{\operatorname{o}}_{h\to 0}(1)= \operatorname{O}\left(\frac{t}{h}\right)+{\operatorname{o}}_{h\to 0}(1).
\end{align*}
By taking $h$ arbitrary small while $t=h^2$, we prove that $\mathcal{H}^{d-1}(\Gamma_n) \to \mathcal{H}^{d-1}(\Gamma_\infty)$.
\subsubsection{Proof of Point~\ref{lem:seq_continuity_volume}}
Using the uniform convergence of $b_{\Omega_n}$ to $b_{\Omega_\infty}$, we deduce that for every $ \eps >0$ there exists $N_0\in \N$ such that
$$\quad b_{\Omega_\infty}^{-1}((-\infty,-\eps])\subset b_{\Omega_n}^{-1}((-\infty,0)) \subset b_{\Omega_\infty}^{-1}((-\infty,\eps)),\qquad \forall n \geq N_0.$$
Hence, we get
$$\mathcal{H}^d(b_{\Omega_\infty} \leq -\eps) \leq \mathcal{H}^d(\Omega_n)\leq \mathcal{H}^d(b_{\Omega_\infty}< \eps).$$
By inner regularity of $\mathcal{H}^d$, $\mathcal{H}^d(b_{\Omega_\infty} \leq -\eps)\xrightarrow{\eps \to 0} \mathcal{H}^d(b_{\Omega_\infty} <0)=\mathcal{H}^d(\Omega_\infty)$. Similarly, by outer regularity $\mathcal{H}^d(b_{\Omega_\infty} < \eps) \xrightarrow{\eps \to 0} \mathcal{H}^d(b_{\Omega_\infty} \le  0)=\mathcal{H}^d(\Omega_\infty)$, where we used that $\Omega_\infty$ belongs to $\Oo$.

\subsubsection{Proof of Point~\ref{lem:seq_continuity_homotopy}}
We want to prove that $\Gamma_n$ is isotopic to $\Gamma_\infty$ for $n$ large enough.
We consider
\begin{equation*}
  \varphi_n(t,x):
  \begin{array}[t]{rcl}
    [0,1] \times \Gamma_\infty & \to     & \R^3           \\
    (t,x)                          & \mapsto & x+t(p_n(x)-x).
  \end{array}
\end{equation*}
According to \cref{lem:p_n_det}, $\varphi_n(1,\cdot)=\tau_n$ is a diffeomorphism from $\Gamma_\infty$ onto $\Gamma_n$. Besides, following the proof of \cref{lem:p_n_det}, we easily get that for $t\in (0,1)$, $\varphi_n(t,\cdot)$ is a diffeomorphism onto its image.
\subsection{Proof of \cref{th:lsc}}\label{sec:proofth:lsc}
Suppose that $(\Omega_n)_{n\in\N}\in \Oo^\N$ $R$-converges toward $\Omega_\infty \in \Oo$.
%\subsubsection{Lower semi-continuity of $F_1$}
%Let us fix $\eps>0$. 
Let $0<t<h$ small enough (to be fixed later) and $n$ large enough.

We recall that the unit normal vector to $\Gamma_n$ is given by $\nabla b_{\Omega_n}$ (see \cref{lem:prop_p}). Then, according to \cref{lem:integral_volume_surface}, %one has
\begin{align*}
  F_1(\Omega_n) = & \int_{\Gamma_n} j_1(x,\nabla b_{\Omega_n}(x),H_{\Gamma_n}(p_n(y))) d\mu_{\Gamma_n}(x)                                    \\
  =               & \frac{1}{2h}\int_{U_h(\Gamma_n)} j_1(p_n(y),\nabla b_{\Omega_n}(p_n(y)),H_{\Gamma_n}(p_n(y))) \,  \det(d_{T_n^{-1}(y)}T_n) \, dy.
\end{align*}
Using Lemma \ref{lem:inclusions}, moreover,
\begin{align}
  F_1(\Omega_n)= & \frac{1}{2h}\int_{U_{h-t}(\Gamma_\infty)} j_1(p_n(y),\nabla b_{\Omega_n}(p_n(y)),H_{\Gamma_n}(p_n(y))) \,  \det(dT_n) \, dy                                       \\
  \label{eq:F1_second}
                 & +\frac{1}{2h}\int_{U_{h}(\Gamma_n)\setminus {U_{h-t}(\Gamma_\infty)}} j_1(p_n(y),\nabla b_{\Omega_n}(p_n(y)),H_{\Gamma_n}(p_n(y))) \,  \det(dT_n) \, dy.
\end{align}
The key idea is to prove that all arguments of $j_1$ 
in the first term
 convergence  
%quantities 
 toward their analogues for $n=\infty$
%in $\Omega_\infty$ 
and to ensure that the second term is small for $t$ small.

Let us start with comparing the first term in the right-hand side with $F_1(\Omega_\infty)$. Notice that
\begin{eqnarray*}
  \lefteqn{\frac{1}{2h}\int_{U_{h-t}(\Gamma_\infty)} j_1(p_n(y),\nabla b_{\Omega_n}(p_n(y)),H_{\Gamma_n}(p_n(y))) \,  \det(dT_n) \, dy} \\
  &=& \frac{1}{2(h-t)}\int_{U_{h-t}(\Gamma_\infty)} \frac{2(h-t)}{2h} \frac{\det(dT_n)}{\det(dT_\infty)} j_1(p_n(y),\nabla b_{\Omega_n}(p_n(y)),H_{\Gamma_n}(p_n(y))) \,  \det(dT_\infty) \, dy.
\end{eqnarray*}

%$$
%\frac{1}{2h}\int_{U_{h-t}(\Gamma_\infty)} j_1(p_n(y),\nu(p_n(y))) \,  \det(dT_n) \, dy = \frac{1}{2(h-t)}\int_{U_{h-t}(\Gamma_\infty)} \frac{2(h-t)}{2h} \frac{\det(dT_n)}{\det(dT_\infty)} j_1(p_n(y),\nu(p_n(y))) \,  \det(dT_\infty) \, dy.$$

By Lemma \ref{lem:det_bounded}, we have
\begin{align}
  \label{eq:estimate_h_t}
  \left\| \frac{2(h-t)}{2h} \frac{\det(dT_n)}{\det(dT_\infty)} -1\right\|_{L^\infty(U_{%2
  h}(\Gamma_\infty))} = {\operatorname{o}}_{h\to 0}(1) + \operatorname{O}\left(\frac{t}{h}\right).
\end{align}

Let us now investigate the mean curvature term. Note that this term is slightly technical to handle for two reasons:
\begin{itemize}
  \item the mean curvature $H_{\Gamma_n}$ is defined as the trace of the shape operator, which is itself defined as the differential of the restriction to the hypersurface of $\nabla b_{\Omega_n}$ (see \ref{sec:curv});
  \item the Hessian of $b_{\Omega_n}$   converges only in a weak sense.
\end{itemize}
We will use the following lemma, which is obtained thanks to the chain rule.
%differentials composition rules.
\begin{lemma}[Theorem 4.4 of \cite{delfourTangentialDifferentialCalculus2000}]\label{lem:delff}
  Let $h < r_0$ and $n \in \bar{\N}$. %Then 
  If $\nabla^2 b_{\Omega_n}(x)$ exists for $x\in U_h(\Gamma_\infty)$, then $\nabla^2 b_{\Omega_n}(p_n(x))$ exist and
  \begin{align*}
    %\label{eq:curvature}
    \nabla^2 b_{\Omega_n}(p_n(x))=\nabla^2 b_{\Omega_n}(x) [\Id- b_{\Omega_n}(x) \nabla^2 b_{\Omega_n}(x)]^{-1}.
  \end{align*}
  Besides, one has that $\nabla^2 b_{\Omega_n}(\tau_n^{-1}(p_n(x)))$ exists as well.
\end{lemma}
Notice that the last part of the statement is not explicitly contained in \cite{delfourTangentialDifferentialCalculus2000} but can be obtained by straightforwardly adapting the proof of its Theorem 4.4.

As $\nabla^2 b_{\Omega_n}$ is uniformly bounded on a neighboorhood of $\Gamma_\infty$ and that $ b_{\Omega_n}(x)\leq h$ for $x \in U_{h}(\Gamma_n)$, there exists $C>0$ such that
$$
 \mathrm{ess\,sup}_{x\in U_{h}(\Gamma_n)} \| [\Id- b_{\Omega_n}(x) \nabla^2 b_{\Omega_n}(x)]^{-1}-\Id \| \leq C h,
$$
for $h$ small.
As a consequence, using \cref{lem:trace_curvature} (given in the appendix), one has
\begin{align*}
  H_{\Gamma_n}(p_n(x)) =\operatorname{Tr} \nabla^2 b_{\Omega_n}(p_n(x))= \operatorname{Tr} \nabla^2 b_{\Omega_n}(x)+\operatorname{O}(h).
\end{align*}
Note also that $H_{\Gamma_n}%(p_n(y))
\leq \frac{1}{r_0}$ on $\Gamma_n$.  We can use the uniform continuity of $j_1$ on a compact set to ensure that for $n$ large enough and $n=\infty$,
\begin{align}
  \label{eq:j_1_trace}
   & \frac{1}{2(h-t)}\int_{U_{h-t}(\Gamma_\infty)} j_1(p_n(y),\nabla b_{\Omega_n}(p_n(y)),H_{\Gamma_n}(p_n(y))) \,  \det(dT_\infty) \, dy \notag\\
   & =
  \frac{1}{2(h-t)}\int_{U_{h-t}(\Gamma_\infty)} j_1(p_n(y),\nabla b_{\Omega_n}(p_n(y)),\operatorname{Tr} \nabla^2 b_{\Omega_n}(y
  )) \,  \det(dT_\infty) \, dy+ \operatorname{O}(h).
\end{align}

The next step is to pass to the limit within the integral. Note that, by definition of $R$-convergence,
\begin{align*}
  \left\{ 
    \begin{array}{rll}
      p_n &\xrightarrow[n\to \infty]{}p_\infty & \text{ strongly in } L^\infty(U_{\frac{r_0}{2}}(\Gamma_\infty) ),\\
      \nabla b_{\Omega_n} \circ p_n &\xrightarrow[n\to \infty]{} \nabla b_{\Omega_\infty} \circ p_\infty & \text{ strongly in } L^\infty(U_{\frac{r_0}{2}}(\Gamma_\infty) ),\\
      \operatorname{Tr} \nabla^2 b_{\Omega_n}%(x) 
      &\xrightarrow[n\to \infty]{} \operatorname{Tr} \nabla^2 b_{\Omega_{\infty}%n
      }%(x) 
       & \text{ weak star in } L^\infty(U_{\frac{r_0}{2}}(\Gamma_\infty) ).
    \end{array}
  \right.
\end{align*}
Thus, using for example \cite{MR348582}, we have
\begin{align}
  \liminf_{n\to \infty}& \frac{1}{2(h-t)}\int_{U_{h-t}(\Gamma_\infty)} j_1(p_n(y),\nabla b_{\Omega_n}(p_n(y)),\operatorname{Tr} \nabla^2 b_{\Omega_n}(y%x
  )) \,  \det(dT_\infty) \, dy  \notag \\
  &\geq \frac{1}{2(h-t)}\int_{U_{h-t}(\Gamma_\infty)} j_1(p_\infty(y),\nabla b_{\Omega_\infty}(p_\infty(y)),\operatorname{Tr} \nabla^2 b_{\Omega_\infty}(y%x
  )) \,  \det(dT_\infty) \, dy \notag \\
  &=
  \frac{1}{2(h-t)}\int_{U_{h-t}(\Gamma_\infty)} j_1(p_\infty(y),\nabla b_{\Omega_\infty}(p_\infty(y)),\operatorname{Tr} \nabla^2 b_{\Omega_\infty}(p_\infty(y%x
  ))) \,  \det(dT_\infty) \, dy+ \operatorname{O}(h) \notag \\
  \label{eq:lim_inf_first_term}
  &= F_1(\Omega_\infty)+\bigO (h).
\end{align}
In order to conclude, let us check that the term in line \eqref{eq:F1_second} is small.
Since $j_1$ is continuous on a compact set, it admits a minimum $m_0\in \R$. Let $m_1= \min(0,m_0)\leq0$. Then,
\begin{eqnarray*}
  \lefteqn{\frac{1}{2h} \int_{U_{h}(\Gamma_n)\setminus {U_{h-t}(\Gamma_\infty)}} j_1(p_n(y),\nabla b_{\Omega_n}(p_n(y)),H_{\Gamma_n}(p_n(y))) \,  \det(dT_n) \, dy}\\
  &\geq & \frac{1}{2h}\int_{U_{h}(\Gamma_n)\setminus {U_{h-t}(\Gamma_\infty)}} m_1 \,  \det(dT_n) \, dy\\
  &\geq &
  \frac{1}{2h}\int_{U_{h+t}(\Gamma_\infty)\setminus {U_{h-t}(\Gamma_\infty)}} m_1 \,  %\frac{2(h \pm t)}{2h}
  \frac{\det(dT_n)}{\det(dT_\infty)} \det(dT_\infty) \, dy.  
%  \frac{1}{2(h \pm t)}\int_{U_{h+t}(\Gamma_\infty)\setminus {U_{h-t}(\Gamma_\infty)}} m_1 \,  \frac{2(h \pm t)}{2h}\frac{\det(dT_n)}{\det(dT_\infty)} \det(dT_\infty) \, dy.
\end{eqnarray*}
Using
$$ \left\| %\frac{2(h\pm t)}{2h}
\frac{\det(dT_n)}{\det(dT_\infty)} -1 \right\|_{L^\infty(U_{2h}(\Gamma_\infty))}= {\operatorname{o}}_{h\to 0}(1)
%+\operatorname{O}\left(\frac{t}{h}\right)
$$
and
$$%\frac{1}{2(h\pm t)}
\int_{U_{h\pm t}(\Gamma_\infty)} m_1 \, \det(dT_\infty) \, dy=2(h\pm t) m_1 \mathcal{H}^{d-1}(\Gamma_\infty),$$
we get
\begin{align}
  \label{eq:second_term}
  \frac{1}{2h} \int_{U_{h}(\Gamma_n)\setminus {U_{h-t}(\Gamma_\infty)}} j_1(p_n(y),\nabla b_{\Omega_n}(p_n(y)),H_{\Gamma_n}(p_n(y))) \,  \det(dT_n) \, dy \geq 
  m_1 \mathcal{H}^{d-1}(\Gamma_\infty) \left({\operatorname{o}}_{h\to 0}(1)+\operatorname{O}\left(\frac{t}{h}\right)\right).
\end{align}
Finally, combining 
Equations~\eqref{eq:j_1_trace}--\eqref{eq:second_term},
%\cref{eq:estimate_h_t,eq:j_1_trace,eq:lim_inf_first_term,eq:second_term}, 
we obtain
$$\liminf_{n\to +\infty}  F_1(\Omega_n) \geq 
(F_1(\Omega_\infty)+O(h)) \left(1%+ {\operatorname{o}}_{h\to 0}(1) 
+ \operatorname{O}\left(\frac{t}{h}\right)\right) + m_1 \mathcal{H}^{d-1}(\Gamma_\infty) \left({\operatorname{o}}_{h\to 0}(1)+\operatorname{O}\left(\frac{t}{h}\right)\right).$$
Hence, taking $h\to 0$ while ensuring $t=\operatorname{o}(h)$ gives
$$\liminf_{n\to +\infty} F_1(\Omega_n) \geq F_1(\Omega_\infty),$$
and finishes the proof.

\subsection{Proof of \cref{th:elliptic_in_boundary}}

Let $(\Omega_n)_{n\in\N}$ denote a sequence that $R$-converges to $\Omega_\infty$, and let $v_n$  denote the unique solution $v_{\Gamma_n}$ to Problem~\eqref{minNRJEpOm} %whenever
for $\Omega=\Omega_n$.
The difficult part here is that $v_{\Gamma_n}$ is not defined on $\Gamma_\infty$. Our main tool will be $\tau_n$, the restriction to $\Gamma_\infty$ of the orthogonal projection $p_n$ on $\Gamma_n$. Those objects were introduced in \cref{subsec:extruded_surf} and we proved that $\tau_n$ is a diffeomorphism between $\Gamma_\infty$ and $\Gamma_n$ in \cref{lem:p_n_det}.

We also have to be careful when we transport the tangential gradient of a function. In order to relate the tangential gradient and the ambient gradient, we establish the following pointwise estimate.

\begin{lemma}
  \label{lem:gradient_relation}
  Let 
  %  $(f_n)_{n\in\N}$ denote a sequence of functions 
 $n\in\N$ and $f_n\in H^1(\Gamma_n)$. Then $f_n\circ \tau_n \in H^1(\Gamma_\infty)$ and, for almost every $x\in \Gamma_\infty$,
  \begin{align}\label{eq:C_n}
    \nabla_{\Gamma_\infty} (f_n \circ \tau_n)(x)= \nabla_{\Gamma_n} f_n%\circ 
    (\tau_n(x)) (\Id +C_n(x)),
  \end{align}
  where 
  %the notation $\nabla f_n\circ \tau_n(x)$ 
 tangential gradients are understood as $d$-dimensional line vectors and 
%, with 
%  must be understood as a line vector, with
  \begin{align*}
        C_n(x)= (\nabla b_{\Omega_n}(x)^\top  \nabla b_{\Omega_n}(x)-\Id)\nabla b_{\Omega_\infty}(x)^\top  \nabla b_{\Omega_\infty}(x) +
    b_{\Omega_n}(x) \nabla^2 b_{\Omega_n}(x) (\nabla b_{\Omega_\infty}(x)^\top  \nabla b_{\Omega_\infty}(x)-\Id).
  \end{align*}
  Besides, $C_n$ converges  toward zero in  the $L^\infty$ norm:  
%  we have the \ms{$L^\infty$} %uniform 
%  convergence of $C_n$ toward zero:
  \begin{align}
    \label{eq:cvg_cn}
    \mathrm{ess\,sup}%\sup
    _{x\in \Gamma_\infty} \| C_n(x)\| \xrightarrow[n \to \infty]{} 0.
  \end{align}
\end{lemma}
\begin{proof}
  First notice that
  $$
    \nabla_{\Gamma_\infty} (f_n \circ p_n)(x)=\nabla (f_n \circ p_n \circ p_\infty)(x)
  $$
  for almost every $x\in \partial\Omega_\infty$, since the directional derivative of $f_n \circ p_n \circ p_\infty$ at the point $x$ in the direction $\nabla b_{\Omega_\infty}(x)$ is zero. 
  By Lemma~\ref{lem:delff}, $\nabla^2 b_{\Omega_n}(x)$ is well-defined for almost every $x$ in $\Gamma_\infty$.
    %Using the differentials composition rules, 
  By Lemma~\ref{lem:prop_p} and the chain rule we obtain, almost everywhere on $\Gamma_\infty$, 
  \begin{align*}
    \nabla (f_n \circ p_n \circ p_\infty)(x) & =((\nabla f_n)\circ p_n)
    % (\nabla (f_n\circ p_n))_{p_n(x)} 
    (\Id-\nabla b_{\Omega_n}^\top  \nabla b_{\Omega_n} -b_{\Omega_n} \nabla^2 b_{\Omega_n}) (\Id-\nabla b_{\Omega_\infty}^\top  \nabla b_{\Omega_\infty} -b_{\Omega_\infty} \nabla^2 b_{\Omega_\infty})            \\
                                             & =((\nabla_{\Gamma_n} f_n) \circ \tau_n)%(\nabla (f_n\circ p_n))_{p_n(x)}
                                              ( \Id -(\Id-\nabla b_{\Omega_n}^\top  \nabla b_{\Omega_n})\nabla b_{\Omega_\infty}^\top  \nabla b_{\Omega_\infty} -b_{\Omega_n} \nabla^2 b_{\Omega_n}(\Id-\nabla b_{\Omega_\infty}^\top  \nabla b_{\Omega_\infty})),
  \end{align*}
  where we %dropped the $x$ dependence for clarity and 
  used that $\nabla f_n =\nabla_{\Gamma_n} f_n$, $p_n=\tau_n$, and $b_{\Omega_\infty}=0$ on $\Gamma_\infty$. This shows \cref{eq:C_n}.

  Let us now bound %the error 
  the $L^\infty$ norm of $C_n$. There exists $C>0$ such that, for every $n$ satisfying $\Gamma_\infty \subset U_{\frac{r_0}{2}}(\Gamma_n)$, %we have
  \begin{align*}
    \mathrm{ess\,sup}%\sup
    _{x \in \Gamma_\infty} \| \nabla^2 b_{\Omega_n}(x) (\nabla b_{\Omega_\infty}^\top (x)  \nabla b_{\Omega_\infty}(x)-\Id)\|  \leq C.
  \end{align*}
  Besides, $\|b_{\Omega%\Gamma
  _n}\|_{L^\infty(\Gamma_\infty)}$ converges to zero. Finally, using the uniform convergence of $\nabla b_{\Omega_n}$ toward $\nabla b_{\Omega_\infty}$, we get
  $$\nabla b_{\Omega_n}^\top  \nabla b_{\Omega_n}\nabla b_{\Omega_\infty}^\top  \nabla b_{\Omega_\infty}
    \xrightarrow[n\to \infty]{L^\infty(\Gamma_\infty)}
    \nabla b_{\Omega_\infty}^\top  (\nabla b_{\Omega_\infty}\nabla b_{\Omega_\infty}^\top ) \nabla b_{\Omega_\infty}=\nabla b_{\Omega_\infty}^\top  \nabla b_{\Omega_\infty}.$$
  %Noting that $\nabla (f_n\circ \tau_n)=\nabla_{\Gamma_n}f_n\circ \tau_n$, 
  This concludes the proof of \cref{eq:cvg_cn}.
\end{proof}
From the solution $v_n$ in $H^1_*(\Gamma_n)$, we introduce the function $w_n$ defined on $\Gamma%\partial\Omega
_\infty$ by 
\begin{align}
  w_n=v_n\circ \tau_n-\frac{1}{ \mathcal{H}^{d-1}(\Gamma_\infty) }\int_{\Gamma%\partial\Omega
  _\infty}v_n\circ \tau_n \, d\mu_{\Gamma%\partial\Omega
  _\infty}.
\end{align}
Note that, defined as such, $w_n$ belongs to $H^1_*(\Gamma_\infty)$.

\paragraph{Step 1: convergence of $(w_n)_{n\in\N}$.}
%\rr{j'ai un peu remanier l'ordre pour faire sans le lemme. Si ca ne vous convient pas l'ancienne version est après le end document.}

Let us start by considering the  sequence of energies $(\mathscr{E}_{\Gamma_n}(v_n))_{n\in\N}$. This sequence is upper bounded %above 
by $0$, since $\mathscr{E}_{\Gamma_n}(v_n)\le \mathscr{E}_{\Gamma_n}(0)=0$ for every $n$. By using the uniform Poincar\'e inequality stated in \cref{lem:PoincaSurf} combined with the Cauchy--Schwarz inequality, we get that $(\int_{\Gamma_n}|v_n|^2\, d\mu_{\Gamma_n}(x))_{n\in\N}$ is bounded. We now compute
\begin{align*}
  \frac{1}{ \mathcal{H}^{d-1}(\Gamma_\infty) }\int_{\Gamma_\infty}v_n\circ \tau_n\, d\mu_{\Gamma_\infty}&= \frac{1}{ \mathcal{H}^{d-1}(\Gamma_\infty) }\int_{\Gamma_n}v_n\, \operatorname{Jac}(\tau_n) d\mu_{\Gamma_n}\\
  &=\left( \frac{\sqrt{\mathcal{H}^{d-1}(\Gamma_n)}}{ \mathcal{H}^{d-1}(\Gamma_\infty)}\|v_n\|_{L^2(\Gamma_n)} \right) \smallo_{n\to\infty} (1),
\end{align*}
where we used Lemma~\ref{lem:p_n_det}, the Cauchy--Schwarz inequality, and the fact that $v_n$ has zero average on $\Gamma_n$. 
Hence, we infer that $w_n=v_n\circ \tau_n+ \smallo_{n\to\infty} (1)$.
Besides, by performing a change of variable and by using \cref{lem:p_n_det,lem:gradient_relation}, we get
\begin{align*}
  \int_{\Gamma_n} \left( \frac12 |\nabla_{\Gamma_n} v_n(y)|^2- f(y)v_n(y) \right)\, d\mu_{\Gamma_n}(y)
  &=\int_{\Gamma_\infty} \left( \frac12 |\nabla_{\Gamma_n} v_n (\tau_n(y))%\circ \tau_n
  |^2- f(\tau_n(y))(v_n\circ \tau_n)%w_n
  (y) \right)\, \operatorname{Jac}(\tau_n)^{-1}d\mu_{\Gamma_\infty}(y)\\
  &=\int_{\Gamma_\infty} \left( \frac12 |\nabla_{\Gamma_\infty} w_n(y)|^2- f(\tau_n(y))w_n(y) \right)\, d\mu_{\Gamma_\infty}(y) + \smallo_{n\to\infty} (1),
\end{align*}
where we used that $\nabla_{\Gamma_\infty} w_n=\nabla_{\Gamma_\infty} (v_n\circ \tau_n)$ by definition of $w_n$.
By using \cref{lem:PoincaSurf} and again the Cauchy--Schwarz inequality, we successively infer that the sequences $(\int_{\Gamma_\infty}|w_n|^2\, d\mu_{\Gamma_\infty}(x))_{n\in\N}$ and $(\int_{\Gamma_\infty}|\nabla_{\Gamma_\infty} w_n|^2\, d\mu_{\Gamma_\infty}(x))_{n\in\N}$ are bounded.
By using \cref{th:RKSurf}, the sequence $(w_n)_{n\in\N}$ converges up to a subsequence toward $w_\infty\in H^1_*(\Gamma_\infty)$, weakly in $H^1(\Gamma_\infty)$ and strongly in $L^2(\Gamma_\infty)$. 
%Denoting similarly a sequence and its subsequence with a slight abuse of notation, 
Up to extracting a subsequence, we get
\begin{align*}
  \int_{\Gamma_\infty} |\nabla_{\Gamma_\infty} w_\infty(x)|^2 \, d\mu_{\Gamma_\infty}&\leq \liminf_{n\to +\infty}\int_{\Gamma_\infty} |\nabla_{\Gamma_\infty} w_n(x)|^2 \, d\mu_{\Gamma_\infty},\\
    \lim_{n\to \infty}\int_{\Gamma_\infty} w_n(x)f(\tau_n(x)) \, d\mu_{\Gamma_\infty}&=\int_{\Gamma_\infty} w_\infty(x)f(x) \, d\mu_{\Gamma_\infty}.
\end{align*}
As a consequence,
$$
\mathscr{E}_{\Gamma_\infty}(w_\infty)\leq \liminf_{n\to +\infty} \mathscr{E}_{\Gamma_n}(v_n).
$$
\paragraph{Step 2: Minimality of $w_\infty$.}
%Let us now check the minimality of the Dirichlet energy $\mathscr{E}_{\Gamma_\infty}(w_\infty)$.
Let $u\in H^1_*(\Gamma_\infty)$ be given and define $z_n$ in $H^1_*(\Gamma_n)$ by
\begin{align}
  z_n=u\circ \tau_n^{-1}-\frac{1}{ \mathcal{H}^{d-1}(\Gamma_n) }\int_{\Gamma_n}u\circ \tau_n^{-1}\, d\mu_{\Gamma_n}.
\end{align}
Let $n\in \N$. By minimality, one has
$$
  \mathscr{E}_{\Gamma_n}(v_n)\leq \mathscr{E}_{\Gamma_n}(z_n).
$$ 
By mimicking the arguments and computations of the first step, we easily get that
\begin{equation}\label{m0918}
 \mathscr{E}_{\Gamma_n}(z_n)=\mathscr{E}_{\Gamma_\infty}(u) +\smallo_{n\to\infty} (1),
\end{equation}
yielding at the end $ \mathscr{E}_{\Gamma_\infty}(w_\infty)\leq \mathscr{E}_{\Gamma_\infty}(u)$.
We infer that $w_\infty$ is the unique solution to the variational problem~\eqref{minNRJEpOm}.
Since the reasoning above holds for any closure point of $(w_n)_{n\in\N}$, it follows that the whole sequence $(w_n)_{n\in\N}$ converges toward $w_\infty$, weakly in $H^1(\Gamma_\infty)$ and strongly in $L^2(\Gamma_\infty)$.
Finally, using $u=w_\infty$ in \eqref{m0918}, we obtain that 
$$\mathscr{E}_{\Gamma_\infty}(w_\infty)= \liminf_{n\to +\infty} \mathscr{E}_{\Gamma_n}(v_n).$$
In particular $(\|w_n\|^2_{H^1(\Gamma_\infty)})_{n\in\N}$ converges toward $\|w_\infty\|^2_{H^1(\Gamma_\infty)}$ which implies the strong convergence of $w_n$ in $H^1(\Gamma_\infty)$.
\paragraph{Step 3: lower semi-continuity of $F_2$.}
Let us use the same notations as previously. %Following exactly the same reasonings as in the first step above, we get that
Using a change of variable, we get
\begin{eqnarray*}
  F_2(\Omega_n) & = & \int_{\Gamma_n} j_2(x,\nabla b_{\Omega_n}(x),v_{n}(x),\nabla_{\Gamma_n} v_{n}(x)) \, d\mu_{\Gamma_n}(x)\\
  & = &  \int_{\Gamma_\infty} j_2(\tau_n(x),\nabla b_{\Omega_n}(\tau_n(x)),v_n(\tau_n(x)),\nabla_{\Gamma_n}v_n \circ \tau_n(x)) \, \operatorname{Jac}(\tau_n)^{-1} d\mu_{\Gamma_\infty}(x).
\end{eqnarray*}
Besides, according to the results above and \cref{lem:prop_p}, the following convergences hold
\begin{align*}
  \left\{ 
    \begin{array}{rll}
      \operatorname{Jac}(\tau_n)^{-1}&\xrightarrow[n\to \infty]{} 1& \text{ strongly in } L^\infty(\Gamma_\infty )\\
      \tau_n &\xrightarrow[n\to \infty]{}\Id|_{\Gamma_\infty} & \text{ strongly in } L^\infty(\Gamma_\infty )\\
      \nabla b_{\Omega_n} \circ \tau_n &\xrightarrow[n\to \infty]{} \nabla b_{\Omega_\infty} & \text{ strongly in } L^\infty(\Gamma_\infty )\\
      v_n\circ \tau_n &\xrightarrow[n\to \infty]{} w_{\infty} & \text{ strongly in } L^2(\Gamma_\infty )\\
      \nabla_{\Gamma_n} v_n\circ \tau_n &\xrightarrow[n\to \infty]{} \nabla_{\Gamma_\infty} w_\infty & \text{ strongly in } L^2(\Gamma_\infty ),
    \end{array}
  \right.
\end{align*}
where $w_\infty$ is the unique solution to the variational problem~\eqref{minNRJEpOm}.

By applying \cite[Theorem~1]{MR348582}, one has
$$
  \liminf_{n\to +\infty}F_2(\Omega_n)\geq F_2(\Omega_\infty).
$$
This is the desired conclusion.

\subsection{Main steps in the proof of \cref{th:elliptic_inside}}\label{sec:proof:th:elliptic_inside}
First note that $u_\Omega-g$ solves \cref{eq:PDE_Omega} with source term $h-\Delta g$ and Dirichlet boundary condition. As a consequence, we can reduce our study to the case of homonegeous Dirichlet condition (i.e., $u_\Omega%-g
=0$ on $\Gamma$).

The method relies on a uniform extension property proved by Chenais in \cite{chenaisExistenceSolutionDomain1975} for surfaces satisfying an $\eps$-cone condition, which is weaker than the uniform ball condition.
\begin{lemma}[{\cite[Theorem II.1]{chenaisExistenceSolutionDomain1975}}]
  There exists a positive constant $C$ (depending only on $r_0$ %, $d$ 
  and $D$) such that for every $\Omega \in \Oo$ there exists an extension operator $E_\Omega\in\mathcal{L}(H^2(\Omega),H^2(D))$ satisfying
  \begin{align}
    \label{eq:extension_operator}
    E_\Omega(u)|_{\Omega}=u, \qquad \| E_\Omega\| _{\mathcal{L}(H^2(\Omega),H^2(D))}\leq C.
  \end{align}
\end{lemma}
We will use this lemma to extend the solution of the PDEs to the whole box $D$.
The next step is to find a uniform $H^2$ estimate of the solutions. In our case such an estimate was proved by Dalphin who extended a result for domains with $\mathscr{C}^2$ boundary obtained by Grisvard in \cite{grisvardEllipticProblemsNonsmooth1985}.
\begin{lemma}[{\cite[Proposition 3.1]{dalphinExistenceOptimalShapes2020}}]
  There exists $C>0$ (depending only on $r_0$ and $D$) %$\Oo$) 
  such that for every $\Omega \in \Oo$ and $ f \in H^2(\Omega)\cap H^1_0(\Omega)$, we have
  \begin{align}
    \|f\|_{H^2(\Omega)} \leq C \|\Delta f\|_{L^2(\Omega)}.
  \end{align}
\end{lemma}

As a consequence, we have a uniform $H^2(D)$ estimate on the extension of the solution $u_\Omega$, namely,
\begin{align}
  \label{eq:apriori_estimate_PDE}
  \|E_\Omega(u_\Omega)\|_{H^2(D)}\leq C \|h\|_{L^2(D)},\qquad \forall \Omega \in \Oo.
\end{align}

Let us now consider $\Omega_n \xrightarrow[]{R} \Omega_\infty$. Using \cref{eq:apriori_estimate_PDE}, we get that $(E_{\Omega_n} (u_{\Omega_n}))_{n\in\N}$ is uniformly bounded in $H^2(D)$.
Up to extracting a subsequence, we can assume that
\begin{align}
  \label{eq:convergence_EDP_ustar}
  E_{\Omega_n} (u_{\Omega_n})\xrightarrow[n \to \infty]{} u^*
  \left\{
  \begin{array}{c}
    \text{weakly in $H^2(D)$} \\
    \text{strongly in $H^{1}(D)$}.
  \end{array}
  \right.
\end{align}
The next step is to prove that the restriction to $\Omega_\infty$ of $u^*$ is $u_{\Omega_\infty}$.

To this aim, let us consider an arbitrary compact set $K$ contained in the interior of ${\Omega}_\infty$ and a $\mathscr{C}^\infty$ function  $\varphi$ with compact support included in $K$. For $n$ large enough, $K$ is contained in the interior of ${\Omega}_n$ (see \cref{lem:inclusions}), and, therefore, one has $\varphi \in H^1_0(\Omega_n)$ for such integers $n$. Using the variational
 formulation of the PDE \eqref{eq:PDE_Omega}, we get
\begin{align}
  \int_D \langle \nabla E_{\Omega_n}(u_{\Omega_n}),\nabla \varphi \rangle - f \varphi =0.
\end{align}
Using the density of $\mathscr{C}^\infty$ functions with compact support in $H^1_0(\Omega_\infty)$ and passing to the limit yields that $u^*|_{\Omega_\infty}=u_{\Omega_\infty}$.
\begin{remark}
  In order to replace Dirichlet boundary conditions by Neumann's ones, one can follow similar steps as those leading to Equation \eqref{eq:convergence_EDP_ustar}. Then, by considering the variational formulation with $\varphi \in \mathscr{C}^\infty(D)$ and passing to the limit in
  $$\int_{\Gamma_n} g \partial_{\nu} \varphi \to \int_{\Gamma_\infty} g \partial_{\nu} \varphi,$$
  (consequence of \cref{cor:continuity} if $g\in \mathscr{C}^0(D)$)
  one gets that $u^*|_{\Omega_\infty}=u_{\Omega_\infty}$.

\end{remark}

The last step is to relate $F_3(\Omega_n)$ and $F_3(\Omega_\infty)$. Since the involved functions belong to Sobolev spaces and since one aims at comparing surface integrals with tubular ones, we need a suitable %type of 
uniform trace result.
\begin{lemma}
  \label{lem:uniform_trace}
  There exists $C$ such that for every $h<\frac{r_0}{2}$, every $n\in \bar \N$ and every $f\in H^{1}(U_\frac{r_0}{2}(\Gamma_n))$,
  \begin{align}
    \|f-\tilde f\circ p_n \|_{L^2(U_h(\Gamma_n))} \leq %\sqrt{h}
     C h \|f\|_{H^1(U_\frac{r_0}{2}(\Gamma_n))},
  \end{align}
  where $\tilde f$ denotes the trace of $f$ on $\Gamma_n$.
\end{lemma}
\begin{proof}
  Let $f$ be a smooth function. According to \cref{lem:prop_p}, 
every point  $y\in U_h(\Gamma_n)$ can be written in a unique way as $y=x+t\nabla b_{\Omega_n}(x)$ with $x=p_n(y)\in \Gamma_n$ and $t\in (-h,h)$. Moreover, one has
  \begin{align*}
    |f(x+t\nabla b_{\Omega_n}(x))-f(x)|^2
     & \leq  C^2 \|\partial_{\nabla b_{\Omega_n}(x)} f(x+y \nabla b_{\Omega_n}(x))\|^2_{L^2_y(-\frac{r_0}{2},\frac{r_0}{2})} |t|,
  \end{align*}
%  for any $x\in \Omega_n$,  
where $\partial_{\nabla b_{\Omega_n}}$ stands for the derivative in the direction $\nabla b_{\Omega_n}(x)$ and $C$ is the norm of the continuous embedding of $H^1([-\frac{r_0}{2},\frac{r_0}{2}])$ into the space $\mathscr{C}^{\frac{1}{2}}$ of $\frac{1}{2}$-H\"older continuous functions.
  Hence, using \cref{lem:det_bounded}, we get
  \begin{align*}
    \|f-f\circ p_n \|^2_{L^2(U_h(\Gamma_n))} & =\int_{-h}^h \int_{\Gamma_n} |f(x+t \nabla b_{\Omega_n}(x))-f(x)|^2 \det(dT_n)\, dx dt                                                                          \\
& \leq \int_{-h}^h \int_{\Gamma_n} C^2 \|\partial_{\nabla b_{\Omega_n}} f(x+y \nabla b_{\Omega_n}(x))\|^2_{L^2_y(-\frac{r_0}{2},\frac{r_0}{2})} |t| \det(dT_n)\, dx dt \\
& \leq  C^2 h^{2} \int_{\Gamma_n}\|\partial_{\nabla b_{\Omega_n}} f(x+y \nabla b_{\Omega_n}(x))\|^2_{L^2_y(-\frac{r_0}{2},\frac{r_0}{2})}(1+{\operatorname{o}}_{h\to 0}(1))\, dx           \\
& \leq  C^2 h^{2} \|f\|^2_{H^1(U_\frac{r_0}{2}(\Gamma_n))}(1+{\operatorname{o}}_{h\to 0}(1)).
  \end{align*}
  We conclude thanks to the density of the smooth functions in $H^1$.
\end{proof}
Using that $u_{\Omega_n}$ is uniformly bounded in $H^2(D)$, let us apply \cref{lem:uniform_trace} to $u_{\Omega_n}$ and $\nabla u_{\Omega_n}$. We obtain
\begin{align*}
  \|u_{\Omega_n}-u_{\Omega_n}\circ p_n\|_{L^2(U_h(\Gamma_n))}^2+ \|\nabla u_{\Omega_n}-(\nabla u_{\Omega_n})\circ p_n\|_{L^2(U_h(\Gamma_n))}^2 ={\operatorname{o}}_{h\to 0}(h).
\end{align*}
The end of the proof is similar to the one of \cref{th:lsc} and consists in using the extruded surface approach to prove
\begin{align*}
  \liminf_{n\to +\infty} F_3(\Omega_n)
   & \geq (1+{\operatorname{o}}_{h\to 0}(1))
  \liminf_{n\to +\infty} \frac{1}{2h} \int_{U_h(\Gamma_n)} j_3(x,\nabla b_{\Omega_n}(p_n(x)),E_{\Omega_n}(u_{\Omega_n})(x),\nabla E_{\Omega_n}(u_{\Omega_n})(x)) \, dx \\
   & \geq (1+ {\operatorname{o}}_{h\to 0}(1)) (1 + \operatorname{O}\left(\frac{t}{h}\right)) \liminf_{n\to +\infty} \frac{1}{2(h-t)} \int_{U_{h-t}(\Gamma_\infty)} j_3(x,\nabla b_{\Omega_\infty}(p_\infty(x)),u^*(x),\nabla u^*(x)) \, dx                     \\
   & \qquad+ {\operatorname{o}}_{h\to 0}(1)+\operatorname{O}\left(\frac{t}{h}\right)\\
   & \geq F_3(\Omega_\infty) + {\operatorname{o}}_{h\to 0}(1)+\operatorname{O}\left(\frac{t}{h}\right),
\end{align*}
which concludes the proof.

\section{Conclusion}

In this paper, we have introduced a new method to tackle the existence issue for shape optimization problems under uniform reach constraints on the considered shapes, of the type
$$
\inf_{\Omega\in \Oo} \int_{\partial \Omega} j(x,\nu_{\partial\Omega}(x), B_{\partial\Omega}(x)) \, d\mu_{\partial \Omega}(x).
$$
While several references such as \cite{guoConvergenceBoundaryHausdorff2013,dalphinUniformBallProperty2018,dalphinExistenceOptimalShapes2020} have already addressed similar questions on the same type of problems, we believe that the approaches developed in this paper are on the one hand simpler, but also sufficiently robust to allow easy extension of the results to more general settings.

For example, we believe that minor adaptations of the developed proof techniques allow one to extend our results to the following cases without much effort:

- Under weaker regularity hypotheses, one could think of replacing the continuity assumption by lower semicontinuity on the integrand $j_{\{1,2,3\}}$. Another example would be to assume that $f$ in equation \eqref{PDEmanifold} belongs to $H^{1/2}(D)$ instead of $\mathcal{C}(D)$.

- More general PDEs could be considered \cref{th:elliptic_inside,th:elliptic_in_boundary}. Extension to general elliptic equations associated with differential operators of the kind $\nabla_\Gamma\cdot (\sigma \nabla_\Gamma)$ satisfying a coercivity property should be straightforward. We also believe that our framework allows extensions to nonlinear elliptic PDEs under reasonable assumptions.

- One could consider costs involving the solution of a minimization problem depending on $\Omega$ but not necessarily related to a PDE. Indeed, in the proof of \cref{th:elliptic_in_boundary}, our study of the variational problem does not rely on the underlying PDE. We treated a case involving a convex minimization problem over the set of divergence-free vectors fields on $\partial \Omega$ in \cite{privatOptimalShapeStellarators2022}.

All those generalizations do not seem obvious when using other methods.
\appendix
\begin{center}
\fbox{\Large{\bf Appendix}}
\end{center}
\section{Curvatures of a submanifold}
\label{sec:curv}

\setcounter{definition}{0}
\renewcommand{\thedefinition}{A.\arabic{definition}}
\setcounter{equation}{0}
\renewcommand{\theequation}{A.\arabic{equation}}
\setcounter{remark}{0}
\renewcommand{\theremark}{A.\arabic{remark}}
\setcounter{lemma}{0}
\renewcommand{\thelemma}{A.\arabic{lemma}}

Let us quickly review the definition of the mean curvature %\ms{and Gaussian curvature [WHY?]}
for an oriented $(d-1)$-submanifold of $\R^d$ with $\mathscr{C}^{1,1}$ regularity. To stick with our notation, we %will 
consider the submanifold to be %defined as 
the boundary of some $\Omega \in \Oo$.

\begin{definition}
  The Gauss map is the application which assigns to each $x \in \Gamma=   \partial \Omega$ the direct unit normal vector to $\Gamma$ at $x$. In our setting it can be defined as
  \begin{align*}
    N \colon \Gamma & \to \mathcal{S}^{d-1}       \\
    x                        & \mapsto \nabla b_\Omega(x).
  \end{align*}
  We can now define the following objects:
  \begin{itemize}
    \item the shape operator (or Weingarten map) is the differential of the Gauss map. For every $x\in \Gamma$, the tangent spaces $T_x\Gamma$ and $T_{N(x)}\mathcal{S}^{d-1}$ are equal as linear subspaces of $\R^d$ and, the shape operator at $x$ is self-adjoint where it is defined. See e.g. \cite[Chapter 5]{jostRiemannianGeometryGeometric2017} for a general introduction.
    \item The trace of the shape operator is called the mean curvature and is denoted $H$\footnote{Note that in differential geometry it is common to define the mean curvature as the trace of the shape operator divided by $(d-1)$.}.
    \item The determinant of the shape operator is called the Gauss curvature.
  \end{itemize}
\end{definition}
\begin{remark}
  The Gauss map is $\frac{1}{r_0}$-Lipschitz continuous (see %\cref{lem:prop_p}
  \cref{lem:carac_Oo}), where $r_0$ is the reach of $\Gamma$. Thus, the shape operator is in $L^\infty$ and at almost every $x\in \Gamma$, all the eigenvalues $\kappa_1(x),\dots,\kappa_{d-1}(x)$ of the shape operator
  %satisfy $|\kappa_i| \leq 
  are bounded in modulus by $
  \frac{1}{r_0}$.
\end{remark}

We insist on the fact that $N$ is defined only on $\Gamma$ and thus the shape operator is not defined on $\R^d$ or any tubular neighborhood of $\Gamma$. Nevertheless, we have the following property.
\begin{lemma}
  \label{lem:trace_curvature}
  The mean curvature %, i.e., the trace of the shape operator, 
  coincides with the trace of $\nabla^2 b_\Omega$ on $\Gamma$.
\end{lemma}
\begin{proof}
  Let $x\in \Gamma$ and let $\mathcal{B}$ be an orthonormal basis of $T_x\Gamma$. Using the identification between $T_x\Gamma$ and the tangent hyperplane (see \cref{rmk:identification_tangent_plan}), we obtain that $\{\nabla b_\Omega(x)\}\cup \mathcal{B}$ is an orthonormal basis of $\R^d$. $\nabla b_\Omega$ is constant along the direction $\nabla b_\Omega(x)$ (see e.g. \cite[Theorem 7.8.5.ii]{delfourShapesGeometries2011}). As a consequence, the trace of $\nabla^2 b_\Omega$ and the mean curvature coincide.
\end{proof}

%%%%%%%%%%%%%%%%%%%%%%%%%%%%%%%%%
\section{$R$-convergence: proof of \cref{th:compact}}\label{sec:appendRCV}

The compactness property follows from two facts. First, the Arzel\`a--Ascoli theorem, combined with the fact that every function $b_{\Omega}$, for $\Omega\in \Oo$, is 1-Lipschitz continuous. Second, the reach constraint which imposes a uniform bound on the second derivative of $b_{\Omega}$. These two facts are used in \cite{delfourShapesGeometries2011} and \cite{dalphinUniformBallProperty2018} to get the sequential compactness results used below.

Let $(\Omega_n)_{n\in\N}$ denote a sequence in $\Oo$. By the compactness property of sets of uniformly positive reach proved in \cite[Chapter~6]{delfourShapesGeometries2011},
%We refer to \cite[Chapter~6]{delfourShapesGeometries2011} for a proof of the compactness property of sets of uniformly positive reach, namely a proof of the $C^0$-convergence. 
 it follows 
%This yields 
that, up to a subsequence, 
%$(\Omega_n)_{n\in\N}$ converges toward $\Omega_\infty$ for the $C^0$-convergence of signed distance functions.
$b_{\Omega_n}$ converges to $b_{\Omega_\infty}$ for the $C^0$ topology on $D$.
In \cite{dalphinUniformBallProperty2018} the convergence %property 
is shown to  hold also for the strong $\mathscr{C}^{1,\alpha}$ topology  (for $\alpha<1$) and for the weak $W^{2,\infty}$ topology  in a $r$-tubular neighborhood of $\partial\Omega_\infty$, with $r<r_0$.

As a consequence, $\operatorname{Reach}(%\Omega_\infty
\Gamma_\infty
)\geq r_0$.
In particular, according to \cref{lem:carac_Oo}, $b_{\Omega_\infty}$ is $\mathscr{C}^{1,1}$ on $\overline{U_{r}(\Gamma_\infty)}$.
%
%The  regularity of  follows from the fact that 
%$\operatorname{Reach}(\partial V_\infty)\geq \rmin$, .  
%    Finally, it remains to show that    $S$ is in $\Oad$ by showing that  
%    $d(S_\infty,P)\geq \delta$. This inequality can be rewritten as
%$$
%\sup_{y\in P} b_{V_\infty}(y)\leq -\delta,
%$$ 
%which is a direct consequence of the uniform convergence of $(b_{V_n})_{n\in \N}$.

\section{The Laplace--Beltrami equation on a manifold: proof of \cref{lem:1915}}\label{proof:lem:1915}

\setcounter{theorem}{0}
\renewcommand{\thetheorem}{C.\arabic{theorem}}
\setcounter{equation}{0}
\renewcommand{\theequation}{C.\arabic{equation}}
\setcounter{proposition}{0}
\renewcommand{\theproposition}{C.\arabic{proposition}}

Let $(\partial\Omega, \mathfrak{g})$ denote a closed compact manifold. We explain hereafter how to understand the equation $\Delta_{\partial\Omega}v=h$ in $\partial\Omega$ in a weak sense, whenever $\Omega\in  \Oo$. Indeed, under this assumption, $\partial\Omega$ is a $\mathscr{C}^{1,1}$ submanifold according to \cref{lem:carac_Oo}, not necessarily $\mathscr{C}^{2}$, which justifies why such an equation cannot be understood in a strong sense.

The key ingredient in what follows is the Rellich--Kondrachov lemma, stating the compactness of the embedding $H^1_*(\partial\Omega)\hookrightarrow L^2(\partial\Omega)$.
\begin{theorem}[Rellich--Kondrachov theorem on surfaces]\label{th:RKSurf}
  Let $\Omega\in \Oo$.  Let $(u_n)_{n\in\N}$ denote a sequence in $H^1_*(\partial\Omega)$ such that  $(\int_{\partial\Omega}|\nabla u_n(x)|^2\, d\mu_{\partial\Omega})_{n\in\N}$ is bounded. There exists $u^*\in H^1_*(\partial\Omega)$ such that, up to a subsequence, $(u_n)_{n\in\N}$ converges to $u^*$ weakly in $H^1_*(\partial\Omega)$ and strongly in $L^2(\partial\Omega)$.
\end{theorem}
\begin{proof}
  According to \cite[Th 4.5.ii]{delfourTangentialDifferentialCalculus2000}, since $\partial\Omega$ is $\mathscr{C}^{1,1}$, the $L^2$ norm $\Vert \cdot \Vert_{L^2(\partial\Omega)}$ on the surface $\partial\Omega$ and the $L^2$ norm $L^2(\partial\Omega)\ni u\mapsto \Vert u\circ p_\Omega \Vert_{L^2(U_h(\partial\Omega))}$ on the thickened surface $U_h(\partial\Omega)$ are equivalent whenever $h>0$ is small enough, where
 $p_\Omega(x)$ denotes the orthogonal projection of $x$ onto $\partial\Omega$, that is, $p_\Omega(x)=x-b_\Omega(x)\nabla b_\Omega(x)$, and   
  $U_h(\partial\Omega)=\{x\in \R^d\mid |b_\Omega(x)|<h \text{ and }p_\Omega(x)\in \partial\Omega\}$.
 % , the notation $p_\Omega(x)$ denoting the orthogonal projection of $x$ onto $\partial\Omega$, in other words $p_\Omega(x)=x-b_\Omega(x)\nabla b_\Omega(x)$.

  Similarly, according to \cite[Th 4.7.v]{delfourTangentialDifferentialCalculus2000}, since $\partial\Omega$ is $\mathscr{C}^{1,1}$,
  the norm $\Vert \cdot\Vert_{H^1_*(\partial\Omega)}$ defined as
  $$
    \Vert u\Vert_{H^1_*(\partial\Omega)}^2=\int_{\partial\Omega}|\nabla_{\Gamma} u|^2\, d\mu_{\partial\Omega},
  $$
%  where $\nabla_{\Gamma}$ denotes the tangential gradient, 
and the norm $\Vert \cdot\Vert_{H^1_{U_h(\partial\Omega)}%(\partial\Omega)
}$ given by
  $$
    \Vert u\Vert_{H^1_{U_h(\partial\Omega)}%(\partial\Omega)
    }=\frac{1}{2h}\int_{U_h(\partial\Omega)}|\nabla_{\Gamma} u\circ p_\Omega|^2\, d\mu_{\partial\Omega}
  $$
  are equivalent whenever $h>0$ is small enough.
  We conclude by using the standard Rellich--Kondrakov theorem (see e.g. \cite[Section~9.3]{brezisFunctionalAnalysisSobolev2011}) on the thickened surface $U_h(\partial\Omega)$.
\end{proof}

The following result is a Poincar\'e type lemma, uniform with respect to the chosen surface in the set $\Oo$.

\begin{proposition}[Poincar\'e lemma on a surface]\label{lem:PoincaSurf}
  Let $r_0>0$ and $\Omega\in \Oo$. There exists $C(r_0,D)>0$ such that
  $$
    \forall u\in H^1_*(\Gamma), \quad \int_{\Gamma}|\nabla_{\Gamma} u(x)|^2\, d\mu_{\Gamma}\geq C(r_0,D)\int_{\Gamma}|u(x)|^2\, d\mu_{\Gamma}.
  $$
\end{proposition}
\begin{proof}
  Let $(\Omega_n,v_n)_{n\in\N}$, with $v_n\in H^1_*(\Gamma_n)$, be a minimizing sequence for the problem
  $$
    \inf_{\Omega \in \Oo}\inf_{u\in H^1_*(\Gamma)}\frac{\int_{\Gamma}|\nabla_{\Gamma} u(x)|^2\, d\mu_{\Gamma}}{\int_{\Gamma}|u(x)|^2\, d\mu_{\Gamma}}.
  $$

  Let us argue by contradiction, assuming that
  $$
    \int_{\Gamma_n} |\nabla_{\Gamma_n} v_n(x)|^2 \, d\mu_{\Gamma_n}\leq \frac{1}{n}\quad \text{and}\quad \int_{\Gamma_n} |v_n(x)|^2\, d\mu_{\Gamma_n}=1
  $$
  by homogeneity of the Rayleigh quotient.
  According to \cref{th:compact}, we can assume without loss of generality that $(\Omega_n)_{n\in\N}$ $R$-converges toward $\Omega_\infty\in \Oo$.

  Let $p_n$ denote the orthogonal projection on $\Gamma_n$ and
  % According to \cref{lem:inclusions}, if $n$ is large enough and $h$ small enough, one has
  %  $$
  %   U_{h-t}(\Gamma_\infty)\subset U_{h}(\Gamma _n) \subset U_{h+t}(\Gamma_\infty).
  %   $$
  let us introduce the function $w_n$ defined in $U_h(\Gamma_n)$ for $h$ as in \cref{lem:det_bounded} and $n$ large enough by $w_n=v_n\circ p_n$. We follow exactly the same lines as in the first step of the proof of \cref{th:elliptic_in_boundary}.
  A direct adaptation of the first step of the proof of \cref{th:elliptic_in_boundary} yields
  %Applying \cref{lem:integralEstim} to $w_n$ with $j_2(x)=|x|^2$ yields respectively
  \begin{eqnarray}
    \int_{\Gamma_\infty}|\nabla_{\Gamma_\infty} w_n(y)|^2\, d\mu_{\Gamma_\infty}(y)=\int_{\Gamma_n}\left|\nabla_{\Gamma_n} v_n\right|^2\, d\mu_{\Gamma_n}(x)+\operatorname{o}(1). \label{eq12121}
  \end{eqnarray}

  We infer that
  \begin{eqnarray*}
    \int_{\Gamma_\infty}|\nabla_{\Gamma_\infty} w_n|^2\, d\mu_{\Gamma_\infty}(x) & \leq \frac{1}{n}+\operatorname{o}(1).
  \end{eqnarray*}

  By using \cref{th:RKSurf}, we get that the sequence $(w_n)_{n\in\N}$ converges up to a subsequence toward $w^*\in H^1_*(\Gamma_\infty)$ weakly in $H^1(\Gamma_\infty)$ and strongly in $L^2(\Gamma_\infty)$. 
  %Denoting similarly a sequence and its subsequence with a slight abuse of notation, 
  Up to extracting a subsequence,
  we get
  \begin{eqnarray*}
    \int_{\Gamma_\infty} |\nabla_{\Gamma_\infty} w^*(x)|^2 \, d\mu_{\Gamma_\infty} & \leq & \liminf_{n\to +\infty}\int_{\Gamma_\infty} |\nabla_{\Gamma_\infty} w_n(x)|^2 \, d\mu_{\Gamma_\infty}=0,\\
    \int_{\Gamma_\infty} |w^*(x)|^2 \, d\mu_{\Gamma_\infty}&=&1, \\
    \int_{\Gamma_\infty} w_n(x) \, d\mu_{\Gamma_\infty}&=&0.
  \end{eqnarray*}
  By using the first equality, we get that $w^*$ is constant on $\Gamma$ and we obtain a contradiction with the two last equalities above.
\end{proof}

Let us now prove \cref{lem:1915}.
Let $(u_n)_{n\in\N}$ denote a minimizing sequence for Problem~\eqref{minNRJEpOm}. Since $(\mathscr{E}_{\Gamma}(u_n))_{n\in\N}$ is bounded, and since
$$
  \mathscr{E}_{\Gamma}(u_n)\geq C(d,r_0)\Vert u_n\Vert_{L^2(\Gamma)}^2-\Vert h\Vert_{L^2(\Gamma)}\Vert u_n\Vert_{L^2(\Gamma)}
$$
according to \cref{lem:PoincaSurf}, we infer that $(\Vert u_n\Vert_{L^2(\Gamma)})_{n\in\N}$ is bounded. Since
$$
  \int_{\Gamma} |\nabla_{\Gamma} u_n(x)|^2 \, d\mu_{\Gamma}=\mathscr{E}_{\Gamma}(u_n)+\int_{\Gamma} u_n(x)h(x)\, d\mu_{\Gamma} \leq \Vert h\Vert_{L^2(\Gamma)}\Vert u_n\Vert_{L^2(\Gamma)},
$$
we infer the existence of $u^*\in H^1_*(\Gamma)$ such that, up to a subsequence, $(u_n)_{n\in\N}$ converges weakly in $H^1_*(\Gamma)$ and strongly in $L^2(\Gamma)$.  
%Denoting similarly a sequence and its subsequence with a slight abuse of notation, 
Up to extracting a subsequence, we get
$$
  \inf_{u\in H^1_*(\Gamma)}\mathscr{E}_{\Gamma}(u)
  =\liminf_{n\to +\infty} \mathscr{E}_{\Gamma}(u_n)\geq \int_{\Gamma} |\nabla_{\Gamma} u^*(x)|^2 \, d\mu_{\Gamma}-\int_{\Gamma} u^*(x)h(x)\, d\mu_{\Gamma}=\mathscr{E}_{\Gamma}(u^*)
$$
and the existence follows.
The uniqueness is standard and follows from the strong convexity of the functional $\mathscr{E}_{\Gamma}$.

\section*{Acknowledgements}
This work has been supported by the Inria AEX StellaCage.
The first author were partially supported by the ANR Project ``SHAPe Optimization - SHAPO''.

\bibliographystyle{abbrv} %{unsrt}
\bibliography{shape_opti}

\begin{thebibliography}{10}

\bibitem{MR348582}
L.~D. Berkovitz.
\newblock Lower semicontinuity of integral functionals.
\newblock {\em Transactions of the American Mathematical Society}, 192:51--57,
  1974.

\bibitem{brezisFunctionalAnalysisSobolev2011}
H.~Brezis.
\newblock {\em Functional analysis, {S}obolev spaces and partial differential
  equations}.
\newblock Universitext. Springer, New York, 2011.

\bibitem{chenaisExistenceSolutionDomain1975}
D.~Chenais.
\newblock On the existence of a solution in a domain identification problem.
\newblock {\em J. Math. Anal. Appl.}, 52(2):189--219, 1975.

\bibitem{dalphinUniformBallProperty2018}
J.~Dalphin.
\newblock Uniform ball property and existence of optimal shapes for a wide
  class of geometric functionals.
\newblock {\em Interfaces Free Bound.}, 20(2):211--260, 2018.

\bibitem{dalphinExistenceOptimalShapes2020}
J.~Dalphin.
\newblock Existence of optimal shapes under a uniform ball condition for
  geometric functionals involving boundary value problems.
\newblock {\em ESAIM Control Optim. Calc. Var.}, 26:Paper No. 108, 30, 2020.

\bibitem{delfourTangentialDifferentialCalculus2000}
M.~C. Delfour.
\newblock Tangential differential calculus and functional analysis on a
  {$C^{1,1}$} submanifold.
\newblock In {\em Differential geometric methods in the control of partial
  differential equations ({B}oulder, {CO}, 1999)}, volume 268 of {\em Contemp.
  Math.}, pages 83--115. Amer. Math. Soc., Providence, RI, 2000.

\bibitem{delfourShapesGeometries2011}
M.~C. Delfour and J.-P. Zol\'{e}sio.
\newblock {\em Shapes and geometries}, volume~22 of {\em Advances in Design and
  Control}.
\newblock Society for Industrial and Applied Mathematics (SIAM), Philadelphia,
  PA, second edition, 2011.
\newblock Metrics, analysis, differential calculus, and optimization.

\bibitem{grisvardEllipticProblemsNonsmooth1985}
P.~Grisvard.
\newblock {\em Elliptic Problems in Nonsmooth Domains}, volume~24 of {\em
  Monographs and Studies in Mathematics}.
\newblock {Pitman (Advanced Publishing Program), Boston, MA}, 1985.

\bibitem{guoConvergenceBoundaryHausdorff2013}
B.-Z. Guo and D.-H. Yang.
\newblock On convergence of boundary {H}ausdorff measure and application to a
  boundary shape optimization problem.
\newblock {\em SIAM J. Control Optim.}, 51(1):253--272, 2013.

\bibitem{henrotShapeVariationOptimization2018}
A.~Henrot and M.~Pierre.
\newblock {\em Shape variation and optimization}, volume~28 of {\em EMS Tracts
  in Mathematics}.
\newblock European Mathematical Society (EMS), Z\"{u}rich, 2018.
\newblock A geometrical analysis, English version of the French publication [
  MR2512810] with additions and updates.

\bibitem{jostRiemannianGeometryGeometric2017}
J.~Jost.
\newblock {\em Riemannian geometry and geometric analysis}.
\newblock Universitext. Springer, Cham, seventh edition, 2017.

\bibitem{privatOptimalShapeStellarators2022}
Y.~Privat, R.~Robin, and M.~Sigalotti.
\newblock Optimal shape of stellarators for magnetic confinement fusion.
\newblock {\em J. Math. Pures Appl. (9)}, 2022.

\end{thebibliography}

\end{document}